\documentclass[10pt, oneside, reqno]{amsart}
\usepackage[utf8]{inputenc}
\usepackage[english]{babel}
\usepackage[T1]{fontenc}
\usepackage{amsmath}
\usepackage{amsfonts}
\usepackage{amssymb}
\usepackage{amsthm}
\usepackage{amscd}
\usepackage{soul}
\usepackage{geometry}
\usepackage{enumitem}
\usepackage{cancel}
\usepackage{graphicx}
\usepackage{mathtools}
\usepackage{color}
\usepackage{comment}
\usepackage{csquotes}
\usepackage{mathdots}
\usepackage{appendix}
\usepackage{xcolor}
\usepackage{bm}
\definecolor{marin}{rgb}   {0.,   0.3,   0.7}
\definecolor{rouge}{rgb}   {0.8,   0.,   0.}
\definecolor{sepia}{rgb}   {0.8,   0.5,   0.}
\usepackage[colorlinks,citecolor=marin,linkcolor=rouge,
            bookmarksopen,
            bookmarksnumbered
           ]{hyperref}

\usepackage{booktabs}

\newcommand\N{\mathbb{N}}
\newcommand\T{\mathbb{T}}

\newcommand\R{\mathbb{R}}
\newcommand\C{\mathbb{C}}

\newcommand{\Fc}{\mathcal{F}}
\newcommand{\Gc}{\mathcal{G}}
\newcommand{\dd}{\mathrm{d}}

\newcommand{\enstq}[2]{\left\{#1~\middle|~#2\right\}}

\newcommand\eps{\varepsilon}
\renewcommand{\Re}{\operatorname{Re}}
\renewcommand{\Im}{\operatorname{Im}}

\newcommand\loc{\mathrm{loc}}

\newcommand\Id{\mathrm{Id}}

\newcommand{\tvphi}{\tilde{\varphi}}
\newcommand{\tv}{\tilde{v}}
\newcommand{\tchi}{\tilde{\chi}}
\newcommand{\tpsi}{\tilde{\psi}}

\newcommand{\tPhi}{\tilde{\Phi}}
\newcommand{\tA}{\tilde{A}}
\newcommand{\tB}{\tilde{B}}
\newcommand{\tC}{\tilde{C}}
\newcommand{\tH}{\tilde{H}}
\newcommand{\tR}{\tilde{R}}

\newenvironment{proofof}[2]{\paragraph{\textit{ Proof of #1 #2.}}}{\hfill$\square$}

\theoremstyle{plain}
\newtheorem{theorem}{Theorem} [section]

\newtheorem{lemma}[theorem]{Lemma}

\newtheorem{proposition}[theorem]{Proposition}

\theoremstyle{remark}
\newtheorem{remark}[theorem]{Remark}

\numberwithin{equation}{section}

\title[Splitting methods for GP]{Splitting methods for the Gross-Pitaevskii equation on the full space and vortex nucleation}

\author{Quentin Chauleur}
\address[Quentin Chauleur]{Univ. Lille, CNRS, Inria, UMR 8524 - Laboratoire Paul Painlevé, F-59000 Lille, France. }
\email{quentin.chauleur@inria.fr}

\author{Gaspard Kemlin}
\address[Gaspard Kemlin]{LAMFA, Universit\'e de Picardie Jules Verne and CNRS, UMR 7352, 80039 Amiens, France.}
\email{gaspard.kemlin@u-picardie.fr}

\begin{document}

\begin{abstract}
We prove the convergence in Zhidkov spaces of the first-order Lie-Trotter and
the second-order Strang splitting schemes for the time integration of the
Gross-Pitaesvkii equation with a time-dependent potential and non-zero boundary
conditions at infinity.  We also show the conservation of the generalized mass
and the near-preservation of the Ginzburg-Landau energy balance law. Numerical
accuracy tests performed on a one-dimensional dark soliton corroborate our
theoretical findings. We finally investigate the nucleation of quantum
vortices in two experimentally relevant settings.
\end{abstract}
\thanks{A CC-BY public
copyright license has been applied by the authors to the present
document and will be applied to all subsequent versions up to the
Author Accepted Manuscript arising from this submission. }
\maketitle

\section{Introduction}

We consider the Gross-Pitaevskii equation
\begin{equation}\label{eq:GP} \tag{GP}
    i \partial_t u = \Delta u + \frac1{\varepsilon^2}(1-|u|^2)u + Vu; \quad u(0)=u_0,
\end{equation}
posed on the whole space $\R^d$ for $d \in \left\{1,2,3 \right\}$, with
non-vanishing boundary conditions
\[|u(t,x)| \to 1 \quad \text{when} \quad |x|\to+\infty .\]
 The function $V=V(t,x)$ represents a (possibly) time-dependent real-valued
 potential, and we denote by $\eps>0$ the constant appearing in front of the
 nonlinearity. This equation is also referred in the literature as the nonlinear
 Schrödinger equation with non-vanishing boundary conditions, or the
 Ginzburg-Landau-Schrödinger equation.

The Gross-Pitaevskii equation appears as a fundamental model in condensed matter
physics to describe the dynamics of Bose-Einstein condensates
\cite{Abid03,Ginzburg58,Neu90} or in nonlinear optics \cite{Kivshar98}. It has
been extensively studied in various settings, including simply connected bounded
domains
\cite{BethuelBrezisHelein94,CollianderJerrard98,LinXin99,JerrardSpirn08},
product spaces \cite{DeLaire24,DeLaire25}, on the sphere \cite{Gelantalis12} or
on the plane \cite{BethuelJerrardSmets08}, when quantum vortex states appear as
minimizers of the energy \cite{ChironPacherie23}.

To our knowledge, fewer works deal with the additional presence of a potential
in \eqref{eq:GP}, which acts here as a stirring potential. In the physics
literature
\cite{HuepeBrachet00,Kwak2023}, authors investigate the nucleation of quantum
vortices as the condensate flows in one direction at constant speed around an
obstacle, typically a cylinder. The only theoretical studies in this direction
that we are aware of concerns the one-dimensional case \cite{Maris2003} and the work \cite{LinWei019} in two-dimensional exterior domains. In the case of
small potentials, one should also mention \cite[Section
5.2]{BethuelGravejatSaut2008} where the existence of finite energy solutions are
proved in two dimensions, as well as \cite{PelinovskyKevrekidis2008}.

There have been a considerable amount of work in the last decades concerning the
numerical analysis and simulations of nonlinear Schrödinger-type equations, and
we refer to the review papers \cite{AntoineBaoBesse2013,BaoCai2013} and
references within. Recently, there have been a renew of interest for taking into
account a possibly time-dependent potential $V$ in such models, in particular in
the context of quantum turbulence \cite{Chauleur2025,ChauleurChicireanu2026} or for highly
oscillatory problems \cite{SuZhao2020}. Note that it also enables to efficiently
compute the dynamics of rotational Bose-Einstein condensates
\cite{BaoMarahrens2013}.

Several works \cite{Caliari21,Bao07} numerically investigate the dynamics of
quantum vortices by discretizing \eqref{eq:GP}, but no theoretical convergence
result has been established so far in the analytical framework related to
\eqref{eq:GP}, which is the purpose of the present paper. We derive convergence
estimates for standard splitting schemes in the functional setting associated with
the Ginzburg-Landau energy
\begin{equation} \label{eq:GL} \tag{GL}
  \mathcal{E}(u(t),t) = \int_{\R^d} |\nabla u(t)|^2 +\frac1{2\varepsilon^2} \int_{\R^d}
  \left(1-|u(t)|^2 \right)^2 + \int_{\R^d} V(t) \left(1-|u(t)|^2\right),
\end{equation}
requiring that $|V(t,x)|$, $|\partial_t V(t,x)|\to 0$ as $|x| \to +\infty$. In
fact, solutions to \eqref{eq:GP} are obtained from the Schr\"odinger flow
$ i \partial_t u(t) = - \frac12\partial_u\mathcal{E}(u(t),t)$ and formally satisfy
the energy balance law
\[  \frac{\dd }{\dd t} \mathcal{E}(u(t),t) =
  \int_{\R^d} \partial_t V(t) (1-|u(t)|^2). \]
In particular if $\partial_t V = 0$ the Ginzburg-Landau energy $\mathcal{E}$ is
conserved along the flow of~\eqref{eq:GP}, and we denote
$\mathcal{E}(u(t),t)=\mathcal{E}(u(t))=\mathcal{E}(u_0)$ in such case.

It should be noted that when the parameter $\eps \to 0$ (so that nonlinear
effects are strongly enhanced), one can describe the evolution of quantum
vortices by a reduced Hamiltonian system of differential equations
\cite{CollianderJerrard98,JerrardSpirn08,LinXin99,ovchinnikovGinzburgLandauEquationIII1998}, based on the
variational framework developed in \cite{BethuelBrezisHelein94}. This idea
has been recently exploited in several works \cite{Bao14,Kemlin2026,Bao23} to
efficiently simulate multiple vortex dynamics. Of course, such asymptotic
results remain valid up to the first time of spontaneous nucleation or
collision of quantum vortices, which may happen in finite time. In the current
work, we do not pursue into that direction, as we are interested in vortex
nucleation processes, and we take from now on $\eps=1$. However, we still
implicitly track the dependency of our constants on $\eps$ in our analysis, see
Remark~\ref{rem:influence_eps}, as some numerical simulations performed in
Section~\ref{sec:numerics} deal with fixed~$\varepsilon<1$.

In order to prove precise convergence estimates for the time integration of
\eqref{eq:GP} in the functional framework associated with the Ginzburg-Landau
energy \eqref{eq:GL}, our work presents several novelties which could be easily
adapted to other contexts. First, we prove a number of classical results on the
Cauchy theory for \eqref{eq:GP}. Second, we rigorously adapt the well-known
abstract Lie derivatives framework for nonlinear splitting methods
\cite{hairerGeometricNumericalIntegration2002,kochErrorAnalysisHighorder2013,Lubich2008}
to the case where the linear flow is replaced by an affine flow. This requires a
number of technical adaptations, for which most computations are carefully
performed in Appendix \ref{app:lie}. Finally, we extend this framework to the
case of non-autonomous infinite dimensional systems, possibly nonlinear,
inspired by the recent presentations from
\cite{blanesConciseIntroductionGeometric2025,blanesSplittingMethodsDifferential2024}
in the finite dimensional case. Details on such generalization are provided in
Section~\ref{sec:non_autonomous}.

This paper is organized as follows. In Section~\ref{sec:main_results}, we
introduce the functional framework and the standard splitting methods for
\eqref{eq:GP} and state our main convergence result. In
Section~\ref{sec:stability}, in addition to an extension of the Cauchy theory
for \eqref{eq:GP}, we state and prove stability estimates for the nonlinear and
affine flows. Local error estimates are then derived in
Section~\ref{sec:local_error}, which imply our convergence result. Conservation
of mass and quasi-preservation of energy by the splitting schemes are proven in
Section \ref{sec:consered_quantities}. Finally, numerical simulations in one and
two dimensions are performed in Section~\ref{sec:numerics}.

Throughout all the paper, $C$ denotes a generic positive constant
independent of the underlying parameters, and we specifically denote by
$C(\alpha)$ or $C_{\alpha}$ a constant depending on the parameter~$\alpha$.

\section{Main results} \label{sec:main_results}

\subsection{Functional framework}

One notorious difficulty in the study of \eqref{eq:GP} is the treatment of the
conditions at infinity, which requires an adapted functional framework. Various
approaches have been employed in the past to deal with the Cauchy problem
associated with \eqref{eq:GP}. One consists in working in the Zhidkov
spaces, defined as the closure for the norm
 \[  \| u \|_{X^k} \coloneqq \| u \|_{L^{\infty}} + \sum_{1 \leq |\alpha | \leq
     k } \| \partial^{\alpha} u \|_{L^2}  \]
 of the space of bounded and uniformly continuous functions $u\in
 \mathcal{C}^k(\R^d)$ with $\nabla u \in H^{k-1}(\R^d)$, see for instance
 \cite{Gallo04,Goubet2007}. The linear Schrödinger flow associated with
 \eqref{eq:GP} can then be defined on $X^k(\R^d)$ by the formula
 \begin{equation} \label{eq:linear_flow_Zhidkov}
 \left( e^{-it\Delta}\varphi \right)(t,x) \coloneqq \frac{e^{-id
     \pi/4}}{\pi^{d/2}} \lim_{\delta \to 0} \int_{\R^d} e^{(i-\delta )|z|^2}
 \varphi(x-2\sqrt{t} z) \dd z
\end{equation}
for any $\varphi \in X^k(\R^d)$, $k > d/2$, and for all $t\geq 0$. One can also
directly work on the energy space
\[  E\coloneqq \enstq{u \in H^1_{\loc}(\R^d)}{\nabla u \in L^2(\R^d),\ |u|^2-1
    \in L^2(\R^d)}  \]
associated with the Ginzburg-Landau energy \eqref{eq:GL} on the full space, as in
\cite{Gerard2006}. In particular, we have $E \subset X^1(\R^d)+ H^1(\R^d)$ for
any $d \geq 1$ \cite[Lemma~1]{Gerard2006}. Another method consists in
working around particular stationary states of \eqref{eq:GP}. For instance, in \cite{Gallo2008} the author proved that \eqref{eq:GP} is locally well posed on spaces of the form $\phi+H^1(\R^d)$ for $1 \leq d \leq 3$,
where $\phi$ denotes a regular function of finite energy $\mathcal{E}(\phi) <
+\infty$ such that
\begin{equation} \label{eq:requirements_phi} \tag{FE}
\phi \in \mathcal{C}^3(\R^d) \subset L^{\infty}(\R^d), \quad \nabla \phi \in
H^3(\R^d), \quad \text{and} \quad |\phi|^2-1 \in L^2(\R^d).
\end{equation}
In dimension $d=1$ with $V=0$, there exist traveling waves for \eqref{eq:GP}
for every speed $ |c| < \sqrt{2}$, also known as \textit{dark soliton} and
explicitly given by
\begin{equation} \label{eq:1D_soliton}
 u(t,x)=\phi_c(x+ct), \quad \phi_c(x)=\sqrt{\frac{2-c^2}{2}} \tanh \left(
   \frac{\sqrt{2-c^2}}{2} x  \right) + i \frac{c}{\sqrt{2}},
\end{equation}
so that the profiles $\phi_c$ satisfy \eqref{eq:requirements_phi}. Note in
particular that $\phi_c(+\infty) \neq \phi_c(-\infty)$. In higher dimensions $d
\geq 2$, traveling waves of finite energy of the form
\[ u(t,x)=\phi(x_1+ct,x_2,\ldots,x_d)   \]
also exist for almost all $|c| < \sqrt{2}$, and their asymptotic behavior at
infinity (under some axisymmetric assumptions around axis $x_1$) were formally
derived in \cite{Jones_1982} and rigorously proven in a series of paper
\cite{Gravejat2003,Gravejat2004,Gravejat2004_bis,Gravejat2006}: there are given
(up to a multiplicative constant of modulus one) by
\begin{equation} \label{eq:2D_soliton}
 \phi(x)-1 \underset{|x| \to + \infty}{\sim} \frac{i \vartheta x_1}{x_1^2+(1-c^2/2)x_2^2}
\end{equation}
in dimension two, and by
\begin{equation} \label{eq:3D_soliton}
  \phi(x)-1 \underset{|x| \to + \infty}{\sim} \frac{i \vartheta x_1}{ \left( x_1^2+(1-c^2/2)(x_2^2+x_3^2)  \right)^{3/2}}
\end{equation}
in dimension three, for some constant $\vartheta \in \R$ related to the energy of the traveling wave $\mathcal{E}(\phi)$.
Notice the algebraic decay in \eqref{eq:2D_soliton} for $d=2$ and in
\eqref{eq:3D_soliton} for $d=3$, compared to the exponential convergence towards
both $\phi_c(+\infty)$ and $\phi_c(-\infty)$ in \eqref{eq:1D_soliton} for $d=1$.
We also refer to \cite{ChironScheid2016,ChironScheid2018} for a presentation of
various finite energy solutions to \eqref{eq:GP} in dimension two. From now on we
fix such a finite energy solution $\phi$ satisfying
\eqref{eq:requirements_phi}.

\subsection{Splitting methods} \label{sec:splitting_u}

The linear flow \eqref{eq:linear_flow_Zhidkov} suggests the following operator splitting methods for the time integration of \eqref{eq:GP}, based on the formulation
 \[  \partial_t u = \mathcal{A}(u) + \mathcal{B}(u,t),   \]
 where
 \[ \mathcal{A}(u) = - i \Delta u \quad \text{and} \quad \mathcal{B} (u,t) = -i
   (1-|u|^2 )u -i V(t)u \]
and the solutions to the subproblems
\[ \left| \begin{aligned}
& \ i\partial_t \xi(t) = \Delta \xi(t),  \quad & \xi(0)=\xi_0 \in X^k(\R^d), \\
& \ i\partial_t \zeta(t) = (1-|\zeta(t)|^2 )\zeta (t)+
V(t_0+t)\zeta(t) ,  \quad &\zeta(0)=\zeta_0 \in X^k(\R^d),\
  t_0\in\R,
\end{aligned} \right. \]
for $k>d/2$ and $t \geq 0$. The associated flows are then explicitly given, for any fixed $t_0 \in \R$, by
\[ \left| \begin{aligned}
& \ \xi(t)=\Phi^t_\mathcal{A}(\xi_0)=e^{-it \Delta} \xi_0, \\
& \ \zeta(t)=\Phi^{t,t_0}_\mathcal{B}(\zeta_0)=e^{-it (1-|\zeta_0|^2)}e^{-i\int_{t_0}^{t_0+t} V(s) \dd s}  \zeta_0.  \\
\end{aligned} \right. \]

We now fix a time horizon $T>0$. Let $N \in \N^*$, and denote by $\tau=T/N \leq 1$ the
time discretization step. We also denote by $t_n=n\tau$ for $0 \leq n \leq N$
the discrete time grid, and we define recursively the Lie-Trotter splitting
scheme
\begin{equation}\label{eq:Lie_split_u}
  u_\mathcal{L}^{n+1}= \Phi_\mathcal{L}^{\tau,t_n} (u_\mathcal{L}^n) \coloneqq
  \Phi_\mathcal{B}^{\tau,t_n} \circ \Phi_\mathcal{A}^\tau (u_\mathcal{L}^n),
  \quad  u_\mathcal{L}^0=u_0,
\end{equation}
as well as the Strang splitting scheme
\begin{equation}\label{eq:Strang_split_u}
  u_\mathcal{S}^{n+1}= \Phi_\mathcal{S}^{\tau,t_n} (u_\mathcal{S}^n) \coloneqq
  \Phi_\mathcal{A}^{\frac{\tau}{2}} \circ \Phi_\mathcal{B}^{\tau,t_n} \circ
  \Phi_\mathcal{A}^{\frac{\tau}{2}} (u_\mathcal{S}^n), \quad
  u_\mathcal{S}^0=u_0.
\end{equation}

We can now state the main result of this paper, after recalling that we
equip $\mathcal{C}^\ell([0,T],H^k(\R^d))$ with the norm
\[
  \| w \|_{\mathcal{C}^\ell_TH^k} = \sup_{0\leq \alpha \leq \ell}
  \sup_{t\in[0,T]} \| \partial_t^\alpha w(t) \|_{H^k}.
\]
We refer to Proposition \ref{prop:Cauchy} and Proposition \ref{prop:high_reg} in
Section~\ref{sec:stability} for the required Cauchy theory of the solution $u$
to \eqref{eq:GP}.

\begin{theorem} \label{theorem_convergence}
    Let $1\leq d \leq 3$ and $T>0$. Assume that $u_0=\phi + v_0$ with $v_0 \in
    H^4(\R^d)$, $V~\in~\mathcal{C}^1(\left[0,T\right],H^4(\R^d))$ and $\phi$
    satisfying \eqref{eq:requirements_phi}. We denote by $u$ the unique global
    solution to \eqref{eq:GP} with initial data~$u_0$.

    \medskip
    \textbf{(Lie splitting).} There
    exist $\tau_{\mathcal{L}}$ and $C_{\mathcal{L}}$ such that
    \[ \forall\  0 < \tau \leq \tau_{\mathcal{L}},\quad
      \| u(t_n)-u_\mathcal{L}^n \|_{X^2} \leq C_{\mathcal{L}} \tau  \]
    where $\tau_{\mathcal{L}}$, $C_{\mathcal{L}}$ depend on
    $d,T,\|v_0\|_{H^4},
    \|\phi\|_{X^4},\|V\|_{\mathcal{C}^0_T H^4}$ and $\|V\|_{\mathcal{C}^1_T H^2}$.

    \medskip
    \textbf{(Strang splitting).} If moreover $v_0 \in H^6(\R^d)$,
    $\nabla \phi \in H^5(\R^d)$ and $V \in
    \mathcal{C}^2(\left[0,T\right],H^6(\R^d))$, then there exist
    ${\tau}_{\mathcal{S}}$ and $C_{\mathcal{S}}$ such that
    \[ \forall\ 0 < \tau \leq {\tau}_{\mathcal{S}}, \quad
      \| u(t_n)-u_\mathcal{S}^n \|_{X^2} \leq {C}_{\mathcal{S}} \tau^2  \]
    where $\tau_{\mathcal{S}}$, ${C}_{\mathcal{S}}$ depend on $d,T,\|v_0\|_{
      H^6},\|\phi\|_{X^6},\|V\|_{\mathcal{C}^0_T H^6}$,
    $\|V\|_{\mathcal{C}^1_T H^4}$ and $\|V\|_{\mathcal{C}^2_T H^2}$.
\end{theorem}

This result can be complemented by the remarks below.
\begin{remark}
  We state our convergence result in the space $X^2(\R^d)$ thanks to the
  algebra property of Sobolev spaces which holds for $1\leq d \leq 3$, namely
  that for all $f,g \in H^2(\R^d)$,
  \begin{equation} \label{eq:sobolev_algebra}
    \|f g \|_{H^2} \leq C_d \|f \|_{H^2} \|g\|_{H^2},
  \end{equation}
  and the fact that $H^k \subset X^k$ as soon as $k>d/2$ (see for instance
  \cite[Remark 2.2.]{Gerard2006}). This property is extensively used
  throughout this paper without explicit mention. Note that for $d=1$, the space
  $H^1(\R^d)$ would be sufficient for such property to hold, which decreases the
  requirement for space regularity for both $v_0$ and $V$ in Theorem
  \ref{theorem_convergence}, but we have chosen to make the statement as simple as
  possible. On the other hand, generalizations to higher dimensions $d \geq 4$ are
  straightforward, assuming more space regularity on $v_0$ and $V$. We give the
  statement for $1 \leq d \leq 3$ here as it covers our physical motivations.
\end{remark}
\begin{remark}
  One would be tempted to work, based on the Cauchy theory of
  \cite{BethuelSaut1999}, in the framework $1+H^k(\R^d)$ (taking the finite
  energy solution $\phi=1$), which drastically simplifies upcoming computations.
  Unfortunately, the affine space $1+H^k(\R^d)$ does not contain several
  solutions of physical interest such as, for instance, dark solitons
  \eqref{eq:1D_soliton} in dimension $d=1$, or traveling waves of \eqref{eq:GP}
  formed of two parallel vortices with degree $\pm 1$ \cite{BethuelSaut1999} for
  $d=2$. More precisely, we know from \cite{Gravejat2004} that traveling waves
  of finite energy of \eqref{eq:GP} with speed $0 < |c| < \sqrt{2}$ satisfy $u-1
  \notin L^2(\R^2)$ in view of the asymptotic \eqref{eq:2D_soliton}. Hence we
  rather work in the usual framework with the affine space $\phi+H^k(\R^d)$, as
  described above.
\end{remark}
\begin{remark}
  We assume throughout our analysis that the potential $V$ can be
  integrated in time for clearness purposes. If it is not the case,
  one needs to discretize the quantity $\int_{t_n}^{t_{n}+\tau} V(s) \dd s$, which
  appears in the definition of the nonlinear flow
  $\Phi_{\mathcal{B}}^{\tau, t_n}$, by a quadrature rule of same local error as
  the underlying splitting scheme.
\end{remark}

\subsection{A related equation and associated schemes} \label{sec:related_equation}
In view of \cite{Gallo2008}, we consider the decomposition $u=\phi+v$, so that
the function $v(t) \in H^k(\R^d)$ satisfies the equation
\begin{equation} \label{eq:GP_v}
i \partial_t v = i \partial_t u = \Delta v + \Delta \phi + (1-|\phi+v|^2+V)(\phi+v)
\end{equation}
with initial condition $v_0=u_0-\phi \in H^k(\R^d)$. Equation \eqref{eq:GP_v}
now suggests the two-term splitting integration
\[  \partial_t v = A(v) + B(v,t),   \]
where
\[ A(v) = - i \Delta v - i \Delta \phi \quad \text{and} \quad B(v,t) = -i
  (1-|\phi+v|^2+V(t))(\phi+v) \]
and the solutions to the subproblems
\[ \left| \begin{aligned}
& \ i\partial_t w(t) = \Delta w(t) + \Delta \phi, & \quad w(0)=w_0 \in H^k(\R^d), \\
& \ i\partial_t z(t) = (1-|\phi+z(t)|^2 +V(t_0+t))(\phi+z(t))
, & \quad z(0)=z_0 \in H^k(\R^d),\ t_0\in\R,
\end{aligned} \right. \]
for $k>d/2$ and $t \geq 0$. The associated flows are thus given by
\begin{equation} \label{eq:flows_splitting_v}
\left| \begin{aligned}
& \ w(t)=\Phi^t_A(w_0)=e^{-it \Delta} w_0 + e^{-it\Delta}\phi - \phi, \\
& \ z(t)=\Phi^{t,t_0}_B(z_0)=  e^{-it (1-|\phi+z_0|^2)}e^{-i\int_{t_0}^{t_0+t} V(s) \dd s} (\phi+  z_0) -\phi.
\end{aligned} \right.
\end{equation}
for any $t_0 \in \R$. Note that even if $\phi \notin H^k(\R^d)$ for any $k \in
\N$, the flow $\Phi^t_A$ is well defined on $H^k(\R^d)$, as we know from
\cite[Lemma~3]{Gerard2006} that, for any $\phi\in X^k(\R^d)$,
\begin{equation} \label{eq:continuity_linear_flow_Zhidkov}
\| e^{-it\Delta}\phi - \phi \|_{L^2} \leq C \sqrt{|t|} \| \nabla \phi \|_{L^2}.
\end{equation}
Moreover, since the free Schrödinger flow preserves Sobolev norms and
$\nabla\phi\in H^{k-1}(\R^d)$, we also have
\begin{equation} \label{eq:continuity_linear_flow_Zhidkov_Hk}
\| e^{-it\Delta}\phi - \phi \|_{H^{k}} \leq C(k) \sqrt{|t|} \| \nabla \phi
\|_{H^{k-1}}.
\end{equation}
Since the potential is real-valued, we also directly compute that
\begin{equation} \label{eq:preserved_quantity_Phi_B}
  \frac12 \frac{\dd }{\dd t} |\phi+z(t)|^2 = \Re \left(
    \overline{(\phi+z(t))} \partial_t z(t)  \right) = - \Re\left( i
    (1-|\phi+z(t)|^2 +V(t_0+t))|\phi+z(t)|^2 \right)=0,
 \end{equation}
hence the quantity $|\phi+z(t)|^2$ is preserved along the flow
$\Phi^{t,t_0}_B$ which allows for explicit integration. Such flows are related
to the one introduced in Section~\ref{sec:splitting_u} on Zhidkov spaces thanks
to the next Lemma, whose proof simply follows from the definition of each flow.

\begin{lemma} \label{lem:equivalence_splittings}
For any $\varphi \in H^k(\R^d)$ with $k \geq 2$, denoting $\zeta=\phi+\varphi \in X^k(\R^d)$, we have
\[   \Phi_\mathcal{A}^{t}(\zeta)=  \phi + \Phi_A^{t}(\varphi) \quad \text{and}
  \quad \Phi_\mathcal{B}^{t,t_0}(\zeta) = \phi + \Phi_B^{t,t_0}(\varphi) \]
for any $t,t_0\in \R_+$.
\end{lemma}

With the same notations adopted in Section~\ref{sec:splitting_u}, we define recursively the Lie-Trotter splitting scheme for \eqref{eq:GP_v}
\[  v_L^{n+1}= \Phi_L^{\tau,t_n} (v_L^n) \coloneqq \Phi_B^{\tau,t_n} \circ \Phi_A^\tau (v_L^n), \quad  v_L^0=v_0, \]
as well as the Strang splitting scheme for \eqref{eq:GP_v}
\[  v_S^{n+1}= \Phi_S^{\tau,t_n} (v_S^n) \coloneqq \Phi_A^{\frac{\tau}{2}} \circ \Phi_B^{\tau,t_n} \circ \Phi_A^{\frac{\tau}{2}} (v_S^n), \quad  v_S^0=v_0. \]
Both schemes enjoy the following convergence result in $H^2(\R^d)$.

\begin{proposition} \label{prop:convergence_splitting_v}
  With the same assumptions as in Theorem \ref{theorem_convergence}, it holds
  \[  \forall\ 0 < \tau \leq \tau_L,\quad
    \| v(t_n)-v_L^n \|_{H^2} \leq C_L \tau  \]
  and,
  \[  \forall\ 0 < \tau \leq \tau_S,\quad
    \| v(t_n)-v_S^n \|_{H^2} \leq {C}_S \tau^2, \]
  where the constants $C_L,\tau_L$ depend on $d,T,\|v_0\|_{
      H^4},{\|\phi\|_{X^4}},\|V\|_{\mathcal{C}^0_T H^4}$, and
    $\|V\|_{\mathcal{C}^1_T H^2}$,
    while the constants $C_S,\tau_S$ depend on $d,T,\|v_0\|_{
      H^6},{\|\phi\|_{X^6}},\|V\|_{\mathcal{C}^0_T H^6}$, and
    $\|V\|_{\mathcal{C}^2_T H^2}$.
\end{proposition}

Thanks to Lemma~\ref{lem:equivalence_splittings}, convergence on $(v_L^n)_{n\in\N}$ and
$(v_S^n)_{n\in\N}$ in Sobolev spaces implies in turn convergence of
$(u_\mathcal{L}^n)_{n\in\N}$ and $(u_\mathcal{S}^n)_{n\in\N}$ in Zhidkov spaces,
hence a large part of this paper (namely Section \ref{sec:stability} and
\ref{sec:local_error}) is mainly devoted to the proof of Proposition
\ref{prop:convergence_splitting_v}.

We follow the strategy initiated by \cite{Lubich2008} for nonlinear
Schrödinger-type equations, as it does not rely on underlying conservation laws,
which may not be available in the case of \eqref{eq:GP_v}. In particular, there
is no preservation of the $L^2(\R^d)$-norm for \eqref{eq:GP_v}, and the
energy balance law does not control the
$\dot{H}^1(\R^d)$-norm of the solution, unlike standard defocusing nonlinear
Schrödinger models. We also adopt the abstract framework of Lie derivatives from
\cite{kochErrorAnalysisHighorder2013,Lubich2008}, which allows to reformulate
local errors of splitting schemes in terms of quadrature errors of integrals in
a very elegant way. A direct proof based on usual Taylor expansions, as for
instance performed in \cite{SuZhao2020}  for Lie and Strang splittings for
oscillatory nonlinear Schrödinger equations with a time-dependent potential,
would also be possible but would lead to tedious computations due to the
numerous terms appearing in the nonlinear flow $\Phi_B^{t,t_0}$, which would
ultimately suffer the readability of the paper. We end this section with a
number of important remarks justifying some of the choices we made.

\begin{remark}
One may be tempted to work directly on the splitting schemes associated with
\eqref{eq:GP} in order to prove Theorem \ref{theorem_convergence}. In fact, we
know from \cite[Theorem 1.1]{Gallo2006} that the following dispersive estimate holds on
Zhidkov spaces
\[
  \| e^{-it \Delta} \varphi \|_{X^k} \leq C \left(1+|t|^{\gamma} \right) \|
  \varphi \|_{X^k}, \quad \text{with} \  \gamma = \left\{ \begin{aligned}
				& \frac12  \frac{1}{1+1/d} & \quad \text{if $d$ is even}, \\
        & \frac{1}{4}  \frac{1}{1+1/(2d)} & \quad \text{if $d$ is odd},
      \end{aligned} \right.  \]
and with the constant $C=2$ if ones follows the computations of
\cite{Gallo2006,Gallo2008}. It turns out that this constant accumulates in the
geometric sum of the proof of the convergence estimate of Section
\ref{sec:local_error}, preventing us to conclude thanks to Lady Windermere's Fan
usual argument. Thus we rather work on the related equation \eqref{eq:GP_v} as
the linear Schrödinger flow $e^{-it\Delta}$ defines an isometry on standard
Sobolev spaces $H^s(\R^d)$ for any $s \in \R$, and thus the affine flow $\Phi_A^t$
is well suited thanks to \eqref{eq:continuity_linear_flow_Zhidkov}.
\end{remark}

\begin{remark} \label{lem:affine_flow_to_linear}
Note that the affine flow $\Phi_A^t$ is linear if and only if the finite
energy function $\phi$ is a constant function $c_0$ (with $|c_0|=1$ due to
boundary conditions at infinity). It is in fact an isometry on every
$H^s(\R^d)$ in this very case and, actually, one can compute (see
Appendix \ref{app:zhidkov}) that
\[  e^{it \Delta} c_0 =\frac{e^{-id \pi/4}}{\pi^{d/2}} \lim_{\delta  \to 0}
  \int_{\R^d} e^{(i-\delta )|z|^2} c_0 \dd z = c_0.  \]
\end{remark}

\begin{remark}
A maybe even more natural splitting integration for equation \eqref{eq:GP_v}
would be a three-term splitting, with solutions to the subproblems
\[ \left| \begin{aligned}
      & \ i\partial_t f = \Delta f; \quad f(0)=f_0; \quad
      \Phi_\mathfrak{A}^t(f_0)=e^{-it\Delta}f_0, \\
      & \ i \partial_t g = (1-|\phi+g|^2+V(t_0+\cdot)g)(\phi+g); \quad
      g(0)=g_0, t_0 \in \R; \quad
      \Phi_\mathfrak{B}^{t,t_0}(g_0)=\Phi_B^{t,t_0}(g_0), \\
      & \ i \partial_t h=\Delta \phi       ; \quad h(0)=h_0; \quad
      \Phi_\mathfrak{C}^t(h_0)=h_0-it\Delta \phi.
    \end{aligned} \right. \]
The Schrödinger flow $\Phi_\mathfrak{A}^t$ has now the advantage of being linear
and to satisfy the usual isometry properties on Sobolev spaces, while one can
easily show that continuity and stability properties hold for the flow
$\Phi_\mathfrak{C}^t$, as from \eqref{eq:continuity_linear_flow_Zhidkov_Hk} we
infer
\[ \| \Phi_\mathfrak{C}^t(h) \|_{H^k} \leq \|h\|_{H^k}+t
  \|\Delta\phi\|_{H^k} \quad \text{and} \quad   \| \Phi_\mathfrak{C}^t(f) -
  \Phi_\mathfrak{C}^t(g) \|_{H^k} \leq \| f-g \|_{H^k}\]
for any $k \in \N^*$. For instance, a possible Lie splitting would then writes
\[ \Phi^{\tau,t}_\mathfrak{L}(\varphi)\coloneqq \Phi^\tau_\mathfrak{A}\circ
  \Phi^{\tau,t}_\mathfrak{B} \circ \Phi^\tau_\mathfrak{C} (\varphi).   \]
A major drawback of such three-term splitting integration comes from the fact
that there is no equivalent of Lemma~\ref{lem:equivalence_splittings}
anymore, in the sense that if $\zeta=\phi+\varphi$, in view of Remark
\ref{lem:affine_flow_to_linear} above,
\[ \Phi^{\tau,t}_\mathcal{L}(\zeta) \neq
  \phi+\Phi^{\tau,t}_\mathfrak{L}(\varphi) \quad \text{if} \quad \phi \neq c_0
  \in \C.  \]
Thus we do not pursue into that direction, and we restrict our attention to the
\textit{structure-preserving} splittings $\Phi_L$ and $\Phi_S$ defined
previously.
\end{remark}

\section{Continuity properties} \label{sec:stability}

We rely for the upcoming analysis on the inequalities
\begin{equation} \label{eq:product_Zhidkov}
  \|\phi \varphi \|_{H^k} \leq C_d  \|\phi\|_{X^k} \| \varphi \|_{H^k}
  \quad \text{and} \quad \| \phi+\varphi \|_{X^k} \leq \| \phi \|_{X^k} + C_d
  \|\varphi \|_{H^k}
\end{equation}
for any $k>d/2$, to estimate the product and the sum of a function $\phi \in X^k(\R^d)$ with
a function $ \varphi  \in H^k(\R^d)$, which can be checked directly recalling that $H^k(\R^d) \subset X^k(\R^d) \subset L^{\infty}(\R^d)$ as $k>d/2$.

\subsection{Propagation of high regularity} If $V=0$, the results from
\cite{DeLaire2010,Gallo2008}
provide that for $v_0~\in~H^1(\R^d)$, there exists a unique solution $v \in
\mathcal{C}^0(\R;H^1(\R^d))$ of \eqref{eq:GP_v}. We briefly extend such global
wellposedness results to the case with a potential, highlighting the main
novelties of the proof.

\begin{proposition} \label{prop:Cauchy}
Let $T>0$. Let $v_0 \in H^1(\R^d)$, $V \in \mathcal{C}^0(\left[0,T\right]; H^1(\R^d))$ and $\partial_t V \in \mathcal{C}^0(\left[0,T\right]; L^2(\R^d))$. Then there exists a unique solution $v \in \mathcal{C}^0(\left[0,T\right]; H^1(\R^d))$ of \eqref{eq:GP_v}.
\end{proposition}
\begin{proof}
Local existence and uniqueness for solutions to \eqref{eq:GP_v} can be proven
classically applying \cite[Theorem 4.4.6.]{Cazenave2003} thanks to a combination
of Strichartz estimates, Sobolev embeddings and Banach's fixed point theorem, as
the presence of the linear potential $V$ is somehow harmless, provided that $V
\in \mathcal{C}^0(\left[0,T\right]; H^1(\R^d))$ (see for instance
\cite[Section~4]{DeLaire2010}). Thus we focus on extending such solutions
globally on $\left[0,T\right]$.

From the energy balance law of the Ginzburg-Landau energy \eqref{eq:GL} we have
\begin{equation} \label{eq:dissipation}
\mathcal{E}(\phi+v(t),t) + \int_0^t \int_{\R^d} \partial_t V(s)
(|\phi+v(t)|^2-1) \dd s = \mathcal{E}(\phi+v_0).
\end{equation}
From a combination of Cauchy-Scwharz and Young inequality, we write
\[ \left| \int_{\R^d} V(t) (1-|\phi+v(t)|^2) \right| \leq \| V(t) \|_{L^2} \|
  1-|\phi+v(t)|^2 \|_{L^2} \leq 8 \| V(t) \|_{L^2}^2 +  \frac{1}{8} \|
  1-|\phi+v(t)|^2 \|_{L^2}^2. \]
As the same holds with $\partial_t V$ instead of $V$, we infer from \eqref{eq:dissipation} that
\[  \sup_{t \in \left[0,T\right]} \int_{\R^d} | \nabla \phi + \nabla v(t)|^2 +
  \frac{1}{4} \sup_{t \in \left[0,T\right]} \int_{\R^d}   (1-|\phi+v(t)|^2)^2
  \leq C(T,\mathcal{E}(\phi+v_0),\| V \|_{\mathcal{C}^0_T L^2},\| \partial_t V
  \|_{\mathcal{C}^0_T L^2} ). \]
Note that both quantities in the left hand side of previous inequality are non
negative. In particular from the bound on $\|\nabla \phi + \nabla v(t)
\|_{\mathcal{C}^0_T L^2} $, with a similar use of Cauchy-Scwharz and Young
inequality we infer
\[ \sup_{t \in \left[0,T\right]} \int_{\R^d}| \nabla v(t)|^2 \leq
  C(T,\mathcal{E}(\phi+v_0), \| V \|_{\mathcal{C}^0_T L^2},\| \partial_t V
  \|_{\mathcal{C}^0_T L^2}, \| \nabla \phi \|_{L^2}).\]
This provides a bound on $\| \nabla v \|_{\mathcal{C}^0_T L^2}$.

On the other hand we get a bound on $\|1-|\phi+v(t)|^2 \|_{\mathcal{C}^0_T L^2}
$. Multiplying equation \eqref{eq:GP_v} by $\overline{v}$, integrating on
$\R^d$ and taking the imaginary part, we have
\[ \frac{1}{2} \frac{\dd }{\dd t} \| v(t) \|_{L^2}^2 =\Re \int_{\R^d}
  \overline{v} \Delta \phi + \Re \int_{\R^d} \overline{v} (1-|\phi+v(t)|^2+V(t))
  \phi, \]
hence from Cauchy-Schwarz inequality
\[ \frac{1}{2} \frac{\dd }{\dd t} \| v(t) \|_{L^2}^2 \leq  \| \Delta \phi
  \|_{L^2} \| v(t) \|_{L^2}  + C  \| v(t) \|_{L^2}  \]
with $C=C(T,\mathcal{E}(\phi+v_0),\| V \|_{\mathcal{C}^0_T L^2},\| \partial_t V
\|_{\mathcal{C}^0_T L^2}, \|\phi \|_{L^{\infty}} )$. This provides a bound on $\| v \|_{\mathcal{C}^0_T L^2}$, which ends the proof.
\end{proof}

Given additional regularity on the initial condition $v_0$, the potential $V$
and the finite energy solution $\phi$, higher regularity can be propagated
through time for \eqref{eq:GP_v}. Such properties are known as
\textit{persistence of regularity} results in the literature, see for instance
\cite[Proposition 3.11]{Tao2006}. We state and prove such a result for the sake
of completeness, as Theorem \ref{theorem_convergence} and Proposition
\ref{prop:convergence_splitting_v} require higher regularity.

\begin{proposition} \label{prop:high_reg}
  Let $k>d/2$, $v_0 \in H^k(\R^d)$, $\nabla \phi \in H^{k-1}(\R^d)$ with $\phi$
  satisfying \eqref{eq:requirements_phi} and
  $V~\in~\mathcal{C}^0(\left[0,T\right]; H^k(\R^d))$. Then the unique solution
  $v$ to \eqref{eq:GP_v} satisfies $v \in \mathcal{C}^0(\left[0,T\right];
  H^k(\R^d))$.
\end{proposition}
\begin{proof}
  Let us first compute from general multivariate Leibniz rule that
  \[   \partial^{\alpha} \left(  1-|\phi|^2\right)=- \sum_{\beta \leq \alpha,
      |\beta| \geq 1} \binom{\alpha}{\beta} \partial^{\beta} \phi \
    \partial^{\alpha-\beta} \overline{\phi} \]
  for any $d$-tuple $\alpha$ such that $|\alpha| \geq 1$, thus we infer as $H^1(\R^d) \subset L^4(\R^d)$ for $1 \leq d \leq 3$ that
  \[ \| 1-|\phi|^2 \|_{H^k} \leq \| 1-|\phi|^2 \|_{L^2} + C_{d,k} \left( \| \nabla
      \phi \|_{H^{k-1}}^2 + \| \nabla \phi \|_{H^{k-1}} \| \phi \|_{L^{\infty}}
    \right)  \]
  where the constant $C_{d,k}>0$ is independent of $\phi$. From Duhamel's
  formula for~\eqref{eq:GP_v}, we write
  \[ v(t) = e^{-it\Delta} v_0 + \left( e^{-it\Delta}\phi - \phi \right) -i
    \int_0^t e^{-i(t-s)\Delta} \left( g(s,v(s)) \right) \dd s, \]
  where $g=g_0+g_1+g_2+g_3$ with
  \[  \left| \begin{aligned}
        & g_0 =(1-|\phi|^2 + V) \phi, \\
        & g_1  = -2\Re(\overline{\phi} v) \phi+ (1-|\phi|^2+V) v, \\
        & g_2 = -|v|^2 \phi - 2\Re (\overline{\phi}v)v , \\
        & g_3 = -|v|^2 v,
      \end{aligned} \right. \]
  we infer by isometry of the linear flow and from both
  \eqref{eq:continuity_linear_flow_Zhidkov} and \eqref{eq:product_Zhidkov} that
  \[  \| v(t) \|_{H^k} \leq \| v_0 \|_{H^k} + C \sqrt{t} \| \nabla \phi
    \|_{H^{k-1}} + C_{\phi,V} \int_0^t\left(1+ \| v(s) \|_{L^\infty}^2 \| v(s)
      \|_{H^k} \right) \dd s,  \]
  where $C_{\phi,V}=C(\| \phi \|_{L^\infty}, \| \nabla \phi \|_{H^{k-1}}, \| V
  \|_{\mathcal{C}^0_T H^k } )>0$. From Gronwall's lemma, we get
  \[ \| v(t) \|_{H^k} \leq ( \| v_0 \|_{H^k} + C_{T,\phi,V} ) \exp \left( C
      \int_0^t \| v(s) \|_{L^{\infty}}^2 \dd s  \right)  \]
  for all $0 \leq t \leq T$, hence $v \in \mathcal{C}^0(\left[0,T\right];
  H^k(\R^d))$ performing similarly to the proof of \cite[Theorem
  5.5.1]{Cazenave2003}.
\end{proof}

\subsection{Stability estimates}
We now give some continuity and stability estimates for the flows
\eqref{eq:flows_splitting_v}. First, thanks to
\eqref{eq:continuity_linear_flow_Zhidkov} and by isometry we have, for any
  $k \in \N$, $t\geq 0$ and any
$\varphi, f,g \in H^k(\R^d)$,
\begin{equation} \label{eq:prop_continuity_A}
  \| \Phi_A^t (\varphi) \|_{H^k} \leq  \| \varphi\|_{H^k} + C \sqrt{t} \|
  \nabla \phi\|_{H^{k-1}} \quad \text{and} \quad  \| \Phi_A^t (f)- \Phi_A^t (g)
  \|_{H^k} =  \| f-g\|_{H^k}.
\end{equation}
We are now going to prove analogous properties for the nonlinear flow $\Phi_B^{t,t_0}$.

\begin{proposition} \label{prop_continuity}
    Let $\varphi \in H^k(\R^d)$ for $k >d/2$, then for all $t_0 \in [0,T]$ and
    $0\leq t \leq T-t_0$, we have
    \[  \| \Phi_B^{t,t_0} (\varphi) \|_{H^2} \leq \exp(C t), \]
    where $C>0$ depends on $\| \varphi\|_{H^k}$, $\| \phi\|_{X^k}$ and $\| V \|_{\mathcal{C}^0_T H^k}$.
\end{proposition}
\begin{proof}
    In view of \eqref{eq:preserved_quantity_Phi_B}, we can write the Duhamel
    formula
    \[    \Phi_B^{t,t_0} (\varphi)  = \varphi - i\int_{0}^{t} (1-|\phi +
      \varphi|^2 +V(t_0+s))(\phi+\Phi_B^{s,t_0}(\varphi)) \dd s \]
    and by taking the Sobolev norm we get
    \begin{align*}
      \|\Phi_B^{t,t_0} (\varphi) \|_{H^k} &\leq  \| \varphi \|_{H^k} + C_{d} t \| \phi
      \|_{X^k} (\|1- |\phi|^2 \|_{H^k} +  \| \phi \|_{X^k} \| \varphi \|_{H^k} + \|
      \varphi \|_{H^k}^2 + \| V \|_{\mathcal{C}^0_T H^k} ) \\
      & \quad +C_{d}  (\|1- |\phi|^2 \|_{H^k} +  \| \phi \|_{X^k} \| \varphi \|_{H^k} + \|
      \varphi \|_{H^k}^2 + \| V \|_{\mathcal{C}^0_T H^k} ) \int_{0}^{t}
      \|\Phi_B^{s,t_0}(\varphi) \|_{H^k} \dd s.
    \end{align*}
    We then infer by Gronwall's lemma that
    \[  \|\Phi_B^{t,t_0} (\varphi) \|_{H^k} \leq C_{d,\phi,V}^1 (\| \varphi \|_{H^k} +
      t) \exp( {C}^2_{d,\phi,V}(1+ \| \varphi \|_{H^k}^2) t )  \]
    with constants $C^1_{d,\phi,V}>0$ and ${C}^2_{d,\phi,V}>0$, which gives the result.
  \end{proof}

The following result shows the stability of the flow $\Phi_B^{t,t_0}$.

\begin{proposition} \label{prop:stability}
Let $M_2>0$, $f,g\in H^2(\R^d)$ such that $\|f\|_{H^2},\|g\|_{H^2} \leq M_2$,
and $t_0\geq 0$. Then, there exists $C(M_2)>0$ such that, for $\tau \leq
  1$,
  \[ \| \Phi_B^{\tau,t_0} (f) - \Phi_B^{\tau,t_0} (g)\|_{H^2} \leq e^{\tau C(M_2)}  \|f - g\|_{H^2}. \]
\end{proposition}
\begin{proof}
As in the proof of Proposition \ref{prop_continuity}, $\Phi_B^{\tau,t_0}$ can
also be defined implicitly by
  \[    \Phi_B^{\tau,t_0}(f) = f - i\int_{0}^{\tau} (\Fc+V(t_0+s)) \left(\phi
      + \Phi_B^{s,t_0}( f) \right)\dd s,\]
where $\Fc = 1-|\phi + f|^2$ is preserved by the flow of $t \mapsto
\Phi_B^{t,t_0}$ from \eqref{eq:preserved_quantity_Phi_B}. Defining similarly
$\Gc =  1-|\phi + g|^2$, it holds
 \begin{equation}\label{eq:phiB_stab}
    \begin{split}
      &\Phi_B^{\tau,t_0}( f) - \Phi_B^{\tau,t_0}( g) \\
      & \quad = f-g - i\int_0^\tau
      \Big( (\Fc + V(t_0+s) ) (\phi + \Phi_B^{s,t_0}( f) ) - (\Gc + V(t_0+s) )(\phi + \Phi_B^{s,t_0}( g) ) \Big)\dd s\\
      &\quad = f - g - i\tau \phi(\Fc - \Gc)
      - i\int_0^\tau \Big(\Fc \Phi_B^{s,t_0} (f) - \Gc \Phi_B^{s,t_0} (g) \Big)\dd s  - i \int_0^\tau V(t_0+s) \Big(\Phi_B^{s,t_0} (f) -
      \Phi_B^{s,t_0} (g) \Big)\dd s .
    \end{split}
  \end{equation}
  The first term $f-g$ in the right hand side of \eqref{eq:phiB_stab} is harmless.
  Regarding the second one, since
  \[ \Fc - \Gc =1-|\phi + f|^2 - 1 + |\phi + g|^2 = \Re\Big( 2\overline{\phi}(g-f) + (f+g)\overline{(g-f)}\Big),  \]
  from \eqref{eq:product_Zhidkov} we infer
  \begin{equation}\label{eq:phiB_stab1}
    \|\phi (\Fc - \Gc)\|_{H^2} \leq C_d \| \phi \|_{X^2} (\| \phi \|_{X^2}+M_2)\|f - g\|_{H^2}.
  \end{equation}
   We now deal with the third term in \eqref{eq:phiB_stab} as follows: we write
  \begin{equation}\label{eq:phiB_stab3}
    \int_0^\tau \Big(\Fc \Phi_B^{s,t_0} (f) - \Gc \Phi_B^{s,t_0} (g) \Big)\dd s =
    \int_0^\tau \Big(\Fc \big(\Phi_B^{s,t_0} (f) - \Phi_B^{s,t_0} (g)\big)
     + \big(\Fc - \Gc)\Phi_B^{s,t_0} (g) \Big)\dd s,
   \end{equation}
   and as before, we bound the two terms inside the integral as
   \begin{multline}\label{eq:phiB_stab4}
    \left\|\int_0^\tau \Fc\big(\Phi_B^{s,t_0} (f) - \Phi_B^{s,t_0} (g)\big)\dd s\right\|_{H^2} \\
    \leq C_d \left(\| 1-|\phi|^2 \|_{H^2} + \| \phi \|_{X^2} M_2 + M_2^2
    \right) \int_0^\tau \| \Phi_B^{s,t_0} (f) - \Phi_B^{s,t_0} (g)\|_{H^2}\dd s
  \end{multline}
  and
  \begin{equation}\label{eq:phiB_stab5}
  \begin{aligned}
    \left\|\int_0^\tau \big(\Fc - \Gc)\Phi_B^{s,t_0} (g)\dd s\right\|_{H^2}
    & \leq C_d(\| \phi \|_{X^2}+M_2)\|f-g\|_{H^2}\int_0^\tau\|\Phi_B^{s,t_0} (g)\|_{H^2}\dd s \\
    & \leq C_d(\| \phi \|_{X^2}+M_2)\|f-g\|_{H^2} C(M_2),
    \end{aligned}
  \end{equation}
  for a constant $C(M_2)>0$, where we used Proposition \ref{prop_continuity} and the fact that $s\leq \tau \leq 1$ in the last step. Finally for the fifth term in \eqref{eq:phiB_stab} we directly infer
    \begin{equation}\label{eq:phiB_stab6}
    \left\| \int_0^\tau V(t_0+s)\big(\Phi_B^{s,t_0} (f) - \Phi_B^{s,t_0} (g)\big)\dd s    \right\| \leq C_d  \| V \|_{\mathcal{C}^0_T H^2}  \int_0^\tau \| \Phi_B^{s,t_0} (f) - \Phi_B^{s,t_0} (g)\|_{H^2}\dd s.
  \end{equation}
  Putting \eqref{eq:phiB_stab} together with \eqref{eq:phiB_stab1}-\eqref{eq:phiB_stab3}-\eqref{eq:phiB_stab4}-\eqref{eq:phiB_stab5}-\eqref{eq:phiB_stab6}, we finally get
  \[
    \|\Phi_B^{\tau,t_0} (f) - \Phi_B^{\tau,t_0} (g)\|_{H^2} \leq
    \|f - g\|_{H^2}
    + \tau C(M_2)\|f - g\|_{H^2}
    + C(M_2) \int_0^\tau \|\Phi_B^{s,t_0}(f) - \Phi_B^{s,t_0}(g)\|_{H^2}\dd s
  \]
  where the constant $C(M_2)$ gathers all the previous constants and only
  depends on $M_2$. Applying Gronwall's lemma, we obtain
  \[
    \|\Phi_B^{\tau,t_0} (f)- \Phi_B^{\tau,t_0} (g)\|_{H^2} \leq \big(1 + \tau C(M_2)\big) e^{\tau C(M_2)}
    \|f - g\|_{H^2}
  \]
  which, after recalling that $1 + {\tau C(M_2)} \leq e^{\tau C(M_2)}$, yields
  the sought-after stability estimate.
\end{proof}

\section{Local error and convergence}  \label{sec:local_error}

Now that we have established stability estimates for both flows $\Phi_A$ and
$\Phi_B$, it remains to study the local error in order to show convergence of
the splitting schemes. For the sake of clarity, we focus first in this section
on the autonomous case, that is $V$ is a time-independent potential. Hence, both
flows $\Phi_A^t$ and $\Phi_B^t$ are independent of the starting time $t_0$.
Extension to the non-autonomous case is analyzed in
Section~\ref{sec:non_autonomous}. As mentioned before, we follow here the
approach from \cite{kochErrorAnalysisHighorder2013,Lubich2008}.

\subsection{A primer on Lie derivatives} \label{sec_Lie_derivatives}

Let $F$ be a vector field with domain in $H^1(\R^d)$. We denote by
$\Phi_F^t(\varphi)$ the solution to the differential equation
\[ \frac{\dd}{\dd t} \psi(t)= F(\psi(t)) \]
with initial data $\psi(0)=\varphi \in H^1(\R^d)$. We also define the \textit{Lie derivative} $D_F$ of $F$ by
\[ (D_F G)(\varphi) = G'(\varphi)[F(\varphi)]  \]
for any other Fréchet differentiable vector field $G$ on $H^1(\R^d)$ and where $G'(\varphi)[\psi]$ denotes the Fréchet derivative of $G$ at $\varphi$ in direction $\psi\in H^1(\R^d)$. In particular from the chain rule we directly get
\[ \frac{\dd}{\dd t}G( \Phi_F^t(\varphi))
= G'(\Phi_F^t(\varphi))[F(\Phi_F^t(\varphi))]
= (D_F G)(\Phi_F^t(\varphi)),  \]
and, by induction,
\[ \frac{\dd^k}{\dd t^k}G( \Phi_F^t(\varphi))= (D_F^k G)(\Phi_F^t(\varphi)).  \]
For every vector field $G$ smooth enough, we then set, by developing the Taylor
series of $G (\Phi_F^t(\varphi))$ at $t=0$,
\begin{equation}\label{eq:Taylor_Lie} G (\Phi_F^t(\varphi))
= \sum_{k=0}^{+\infty} \frac{t^k}{k!} \left(\frac{\dd^k}{\dd t^k}G( \Phi_F^t(\varphi))\right)_{t=0}
= \sum_{k=0}^{+\infty} \frac{t^k}{k!} (D_F^kG)(\varphi)
\eqqcolon \left( \exp(t D_F)G \right)(\varphi).  \end{equation}
In particular taking $G=\Id$, we can rewrite our flow as an exponential, namely
\[ \Phi_F^t(\varphi)= \exp(t D_F) \Id(\varphi). \]
We derive from \eqref{eq:Taylor_Lie} the following derivation rule
\[  \frac{\dd}{\dd t}  \exp(t D_F)G (\varphi) = \Big(\exp(t D_F) D_F G \Big)
  (\varphi),\]
as well as the equality
\begin{equation} \label{eq:FoG_flux}
  \Phi_F^t \circ \Phi_G^s (\varphi) = \exp(s D_G) \exp(t D_F) \Id(\varphi),
\end{equation}
where one has to be careful with sense of composition. We finally introduce the
\textit{Lie bracket}
\begin{equation} \label{eq:commutator_vector_fields}
\left[D_F,D_G  \right]=D_F D_G - D_G D_F=D_{\left[ G,F \right]}
\end{equation}
where we remark that the commutator\footnote{We recall that, for vector fields,
  the commutator is defined as $[F,G](\varphi) = F'(\varphi)[G(\varphi)] -
  G'(\varphi)[F(\varphi)]$, which reduces to $[F,G] = FG - GF$ when $F$ and $G$
  are linear operators.} of the vector fields are reversed, see
Lemma~\ref{lem:app-1}. These identities, as well as other manipulation rules,
are discussed in more details in Appendix~\ref{app:lie}. We also see from
\eqref{eq:FoG_flux} that if two vector fields commute, then their respective
flows commute too and $\Phi_{F+G}^t = \Phi_F^t\circ\Phi_G^t$. From the renowned
Baker-Campbell-Hausdorff (BCH) formula
\cite[Section~III.4]{hairerGeometricNumericalIntegration2002}, it appears that
we need to derive commutator estimates to study the error between $\Phi_{F+G}^t$
and $\Phi_F^t\circ\Phi_G^t$, which is the goal of the next section.

\subsection{Commutator estimates}\label{sec:commutator_est}
Here we recall that, for a fixed $\phi$ with finite energy \eqref{eq:requirements_phi}, we are using the
following vector fields, with domains possibly strict subsets of $H^1(\R^d)$,
\[ A(\varphi)=-i\Delta \varphi - i \Delta \phi  \quad \text{and} \quad B(\varphi)=-i(1-|\phi+\varphi|^2+V)(\phi+\varphi)  \]
associated respectively to the flows
\[  \Phi_A^t(\varphi)=e^{-it\Delta}\varphi + e^{-it\Delta}\phi-\phi  \quad
  \text{and} \quad \Phi_B^{t}(\varphi)= e^{-it(1-|\phi+\varphi|^2+V)}
  (\phi+\varphi) - \phi.  \]
Note also that, since $A$ is affine, it holds $(\Phi_A^t)'(\varphi) \equiv
  e^{-it\Delta}$ as well as
\[
  A'(\varphi)[\psi] = -i\Delta \psi,\quad \text{and} \quad A''(\varphi) \equiv
  0.
\]
We additionally directly compute that
\[
  \begin{split}
    B'(\varphi)[\psi] &= 2i \Re \left( \overline{\psi}(\phi+\varphi)
    \right)(\phi+\varphi) - i \psi (1-|\phi+\varphi|^2+V), \\
    B'' (\varphi) [\psi,\chi] &= 2 i\left( \psi\chi(\overline{\phi+\varphi}) +
      2\Re(\overline\psi\chi)(\phi+\varphi) \right),
  \end{split}
\]
from which we infer the following bounds on the nonlinear vector field using
\eqref{eq:sobolev_algebra} and \eqref{eq:product_Zhidkov}.

\begin{lemma} \label{lemma_nonlinear_operator_estimate}
  For any $\varphi$, $\psi$, $\chi \in H^2(\R^d)$, we have
  \[ \begin{split}
      \| B (\varphi) \|_{ H^2} & \leq C_d \left( \| \phi\|_{X^2} + \|\varphi\|_{H^2} \right) \left(1 + \| V\|_{H^2} + \| \phi\|_{X^2}^2 + \|\varphi\|_{H^2}^2 \right)  \\
      \| B' (\varphi) [\psi] \|_{ H^2} & \leq
       C_d \| \psi\|_{H^2} \left(1 + \| V\|_{H^2} + \| \phi\|_{X^2}^2 + \|\varphi\|_{H^2}^2 \right),\\
      \| B'' (\varphi) [\psi,\chi] \|_{ H^2} & \leq
       C_d \| \psi\|_{H^2} \| \chi\|_{H^2} \left( \| \phi\|_{X^2} + \|\varphi\|_{H^2} \right).
    \end{split} \]
\end{lemma}

We can similarly show the following commutator estimates.

\begin{lemma} \label{lemma_commutator_estimate}
  For any $\varphi$, $\psi \in H^4(\R^d)$, we have
  \[ \begin{split}
      \| \left[ A,B \right] (\varphi) \|_{ H^2} & \leq  C_d \left(\| V\|_{H^4}^2 + \| \phi\|_{X^4}^3+\|\varphi \|_{H^4}^3  \right), \\
      \| \left[ A,B \right]' (\varphi) [\psi] \|_{H^2} & \leq C_d \| \psi \|_{H^4} \left(\| V\|_{H^4}^2 + \| \phi\|_{X^4}^2+\|\varphi \|_{H^4}^2  \right) ,
    \end{split} \]
  and for any $\varphi \in H^6(\R^d)$,
  \[\| \left[ A,\left[ A,B\right] \right]  (\varphi) \|_{ H^2} \leq C_d \left(\| V\|_{H^6}^2 + \| \phi\|_{X^6}^3+\|\varphi \|_{H^6}^3  \right) .\]
\end{lemma}

\begin{proof}
 We compute
  \begin{align*}
    [A,B](\varphi) &= A'(\varphi)[B(\varphi)] - B'(\varphi)[A(\varphi)] \\
    & =\left( 2 (\phi+\varphi)(\overline{\Delta\phi+\Delta \varphi})  + 2 |\nabla \phi + \nabla \varphi|^2 -\Delta V  \right) (\phi+\varphi) \\
    & \quad + \left( 4 \Re \left( (\overline{\phi+\varphi}) (\nabla \phi + \nabla \varphi)  \right) - 2 \nabla V  \right) \cdot (\nabla \phi + \nabla \varphi),
  \end{align*}
  and by taking the $H^2$-norm, we get the result thanks to
  \eqref{eq:sobolev_algebra} and \eqref{eq:product_Zhidkov}. First note that
  from the previous computations we get
  \[  \begin{aligned}
      \left[  A,B\right]'(\varphi) \left[  \psi \right]  = &  \left( 2 \psi ( \overline{\Delta \phi + \Delta \varphi}) + 2(\phi+\varphi) \overline{\Delta \psi} + 4 \Re \left(\overline{\nabla \psi} \cdot (\nabla \phi + \nabla \varphi) \right)  \right) (\phi+\varphi) \\
      & +\left( 2( \phi+\varphi)(\overline{\Delta \phi+\Delta \varphi})+2|\nabla \phi + \nabla \varphi |^2 -\Delta V \right) \psi\\
      & + 4 \Re \left( \overline{\psi}(\nabla \phi+ \nabla \varphi) + \nabla \psi (\overline{\phi+\varphi})  \right) \cdot (\nabla \phi+ \nabla \varphi) \\
      & + \left(4\Re ((\overline{\phi+\varphi})(\nabla \phi+\nabla \varphi))- 2 \nabla V \right) \cdot \nabla \psi
    \end{aligned}  \]
  for any regular enough vector field $\psi$. Hence, through tedious computations, we infer that
  \[
  \begin{aligned}
    i\left[ A,\left[ A,B\right] \right]  (\varphi)  =\ & iA'(\varphi) \left[ \left[
        A,B\right](\varphi) \right] - i\left[  A,B\right]'(\varphi) \left[
      A(\varphi) \right] \\
    =\ & 4 \left( |\Delta  \phi+ \Delta \varphi |^2 + {4} (\nabla \phi + \nabla
      \varphi) \cdot (\overline{\nabla \Delta (\phi+\varphi)})
      +\overline{\Delta^2 (\phi+\varphi)}(\phi+\varphi)\right)(\phi +
    \varphi)  \\
    &+4 \left( 2 \nabla (\phi + \varphi) \overline{\Delta (\phi + \varphi)} + 4
      \Re \left( \overline{\nabla (\phi + \varphi)} \Delta (\phi +
        \varphi) \right)  - 4 \nabla \Delta V \right) \cdot \nabla (\phi + \varphi) \\
    &+ 4 \left( 2|\nabla\phi + \nabla\varphi|^2 + 2\Re( \overline{(\phi +
        \varphi)} (\Delta \phi + \Delta\varphi)) - \Delta V \right) (\Delta\phi +
    \Delta\varphi) - \Delta^2 V (\phi+\varphi).
     \end{aligned}
   \]
 Once again, taking the $H^2$-norm and using algebra product rules \eqref{eq:sobolev_algebra} and \eqref{eq:product_Zhidkov} gives the expected estimate.
\end{proof}

\subsection{Taylor formulas and local error} \label{sec:local_error_taylor}
We fix $\tau>0$, and we denote the vector field $H=A+B$ on $H^1(\R^d)$ so that
the solution to \eqref{eq:GP_v} after one time step $\tau$ can be written as
\[ v(\tau)= \exp(\tau D_H) \Id(v_0)   \]
with the notation of Section~\ref{sec_Lie_derivatives}. We can also rewrite
Duhamel's formula as (see Lemma~\ref{lem:app2} for more details)
\begin{equation} \label{eq:affine_Duhamel}
\begin{aligned}
  v(\tau) & = \Phi_A^t(v_0) + \int_0^{\tau} e^{-i(\tau-s)\Delta} B(v(s)) \dd s \\
  & =\exp(\tau D_A) \Id(v_0)  +  \int_0^{\tau} \exp(s D_H) D_B \exp((\tau-s)D_A) \Id(v_0) \dd s.
\end{aligned}
\end{equation}
Note that, while the affine flow $\Phi_A^t(v_0)$ appears outside of the
  integral in the first expression, only its derivative appears inside the
  integral. It is thus remarkable that the Lie affine flow $\exp(t D_A)$ appears
  two times in the second expression, and we refer to the proof of Lemma
\ref{lem:app0} for more details. We can now rewrite in Lie derivatives notations the first numerical iteration associated respectively to the Lie and Strang splitting, namely
\begin{equation}\label{eq:Lie}
  v_L^1= \Phi_B^{\tau} \circ \Phi_A^{\tau}(v_0)= \exp(\tau D_A) \exp(\tau D_B)  \Id(v_0)
\end{equation}
and
\begin{equation}\label{eq:Strang}
  v_S^1= \Phi_A^{\tau/2} \circ \Phi_B^{\tau} \circ \Phi_A^{\tau/2}(v_0)= \exp \left(\frac{\tau}{2} D_A \right) \exp(\tau D_B) \exp \left(\frac{\tau}{2} D_A \right) \Id(v_0).
\end{equation}
We can then prove the following property:

\begin{proposition}[Autonomous case] \label{prop:local_error}
  We have:

    \medskip
    \textbf{(Lie splitting)} $\| v(\tau) - v_L^1 \|_{H^2} \leq C_L\tau^2$,
      where $C_L=C(\|v_0\|_{H^4}, \| V\|_{H^4}, \| \phi \|_{X^4}) >0$,

    \medskip
    \textbf{(Strang splitting)} $\| v(\tau) - v_S^1 \|_{H^2} \leq {C}_S \tau^3$,
      where ${C}_S={C}(\|v_0\|_{H^6}, \| V\|_{H^6}, \| \phi \|_{X^6}) >0$.

\end{proposition}

\begin{proof}
  We begin the proof by iterating formula \eqref{eq:affine_Duhamel}, writing that
\begin{equation} \label{eq_true_sol_expression_local_error}
  v(\tau)= \exp(\tau D_A) \Id(v_0) +	\int_0^{\tau} \exp(s D_A) D_B \exp((\tau-s)D_A) \Id(v_0) \dd s  + R_1
\end{equation}
where
\[ R_1 \coloneqq \int_0^{\tau} \int_0^s \exp(\sigma D_H) D_B \exp((s-\sigma)D_A ) D_B \exp ( (\tau-s)D_A) \Id(v_0) \dd \sigma \dd s.  \]
  We also perform the second order Taylor expansion
  \begin{equation}\label{eq:Taylor_DB}
    \exp(\tau D_B)\Id(v_0) = v_0 + \tau D_B\Id(v_0)
    + \tau^2 \int_0^1 (1-\theta) \exp(\theta \tau D_B)D_B^2\Id(v_0) \dd \theta.
  \end{equation}

  \medskip
 \textbf{(Lie splitting).}
  We first focus on the Lie scheme. Using such Taylor expansion in the definition
  of $v_L^1$, we get by substracting \eqref{eq:Lie} to
  \eqref{eq_true_sol_expression_local_error} that
  \[ v(\tau) - v_L^1 = R_0^{(L)} + R_1 -R_2^{(L)} \]
  where
  \[ R_0^{(L)} \coloneqq  \int_0^{\tau} \exp(s D_A) D_B \exp((\tau-s)D_A) \Id(v_0) \dd s - \tau \exp(\tau D_A) D_B \Id(v_0) \]
  and
  \[ R_2^{(L)} \coloneqq \tau^2 \int_0^1  (1-\theta) \exp(\tau D_A) \exp(\theta \tau D_B)D_B^2 \Id(v_0) \dd \theta . \]
  We see that $R_0^{(L)}$ corresponds to the quadrature error of the right rectangle rule of the function
  \begin{equation}\label{eq:ffunc}
    f(s)= \exp(s D_A) D_B \exp((\tau-s)D_A) \Id(v_0) \\
  \end{equation}
  applied on the interval $\left[0,\tau\right]$. Such error can be written as
  \[ \int_0^{\tau} f(s) \dd s -\tau f(\tau)=  -\int_0^{\tau}  \theta f'(\theta ) \dd \theta.  \]
  We compute
  \begin{align*}
    f'(s)
    & = D_A \exp(s D_A) D_B \exp((\tau-s)D_A) \Id(v_0) - \exp(s D_A) D_B D_A \exp((\tau-s)D_A) \Id(v_0) \\
    & =\exp(s D_A) \left[ D_A,D_B \right] \exp((\tau-s)D_A) \Id(v_0) \\
    & = \exp(s D_A) D_{\left[ B,A \right]} \exp((\tau-s)D_A) \Id(v_0) \\
    & = e^{-i(\tau-s)\Delta}\left( \left[ B,A \right]
        \left(\Phi_A^s(v_0)\right)\right)
  \end{align*}
  which makes appear the commutator between $A$ and $B$ and where the last
  line is obtained with Lemma~\ref{lem:app0}. We can thus estimate the
  quadrature error term using \eqref{eq:prop_continuity_A}, Lemma
  \ref{lemma_commutator_estimate} and the preservation of Sobolev norms by the
  free Schr\"odinger propagator:
    \[
      \begin{split}
      \|e^{-i(\tau-s)\Delta}\left( \left[ B,A
        \right]\left(\Phi_A^s(v_0)\right)\right)\|_{H^2} &=
      \|\left[ B,A \right]\left(\Phi_A^s(v_0)\right)\|_{H^2} \\ &\leq
      C_d\big(\|V\|_{H^4}^2 + \|\phi\|_{X^4}^3 + \|\Phi_A^s(v_0)\|_{H^4}^3
      \big) \\ &\leq
      C_d\big(\|V\|_{H^4}^2 + \|\phi\|_{X^4}^3 + (\|v_0\|_{H^4} +
      \sqrt{s}\|\phi\|_{X^4})^3 \big)
      \end{split}
    \]
  This yields
  \begin{equation}\label{eq:R0_Lie}
    \| R_0^{(L)} \|_{H^2} \leq C_0 (\tau^2+\tau^{7/2} ) \leq C_0 \tau^2 ,
  \end{equation}
  as $\tau \leq 1$, and where $C_0=C(\|v_0\|_{H^4}, \| V\|_{H^4},\|\phi\|_{X^4})>0$. We now
  estimate the remaining terms $R_1$ and $R_2^{(L)}$. We first note that (see
  Lemma~\ref{lem:app3})
  \begin{multline*}
    \exp(\sigma D_H) D_B \exp((s-\sigma)D_A ) D_B \exp (
    (\tau-s)D_A) \Id(v_0) \\
    = e^{-i(\tau-s)\Delta} B' \left( \Phi_A^{s-\sigma}(v(\sigma)) \right)
    \left[e^{-i(s-\sigma)\Delta} B(v(\sigma)) \right].
  \end{multline*}
  Hence, from \eqref{eq:prop_continuity_A} and
  Lemma~\ref{lemma_nonlinear_operator_estimate} we infer that
  \begin{equation}\label{eq:R1_Lie}
    \| R_1 \|_{H^2} \leq C_1 \tau^2
  \end{equation}
  where $C_1$ only depends on $\|v_0\|_{H^2}$, $\| \phi \|_{X^2}$ and $\|V\|_{H^2}$. For
  $R_2^{(L)}$, we remark, with the help of Lemma~\ref{lem:app4}, that
  \[  \exp(\tau D_A) \exp(\theta \tau D_B)D_B^2\Id(v_0) = B'\big(\Phi_B^{\theta
      \tau}(\Phi_A^\tau(v_0))\big)\big[B(\Phi_B^{\theta
      \tau}(\Phi_A^\tau(v_0))) \big]
  \]
  so that we can infer from \eqref{eq:prop_continuity_A}, Lemma~\ref{lemma_nonlinear_operator_estimate} and
  Proposition \ref{prop_continuity} that
  \begin{equation}\label{eq:R2_Lie}
    \| R_2^{(L)} \|_{H^2} \leq C_2 \tau^2,
  \end{equation}
  where once again $C_2$ only depends on $\|v_0\|_{H^2}$, $\| \phi \|_{X^2}$ and
  $\|V\|_{H^2}$. This proves the local error for the Lie splitting.

  \medskip
  \textbf{(Strang splitting)}
  We turn to the Strang splitting, using Taylor expansion \eqref{eq:Taylor_DB} of $ \exp(\tau D_B)$
  in the definition of $v_S^1$ which gives, by substracting \eqref{eq:Strang} to
  \eqref{eq_true_sol_expression_local_error},
  \[ v(\tau) - v_S^1 = R_0^{(S)} + R_1 -R_2^{(S)} \]
  where
  \[ R_0^{(S)} \coloneqq  \int_0^{\tau} \exp(s D_A) D_B \exp((\tau-s)D_A) \Id(v_0) \dd s - \tau \exp \left(\frac{\tau}{2} D_A \right) D_B \exp \left(\frac{\tau}{2} D_A \right) \Id(v_0) \dd s \]
  and
  \[ R_2^{(S)} \coloneqq \tau^2 \int_0^1  (1-\theta) \exp \left(\frac{\tau}{2} D_A \right) \exp(\theta \tau D_B)D_B^2 \exp \left(\frac{\tau}{2} D_A \right) \Id(v_0) \dd \theta . \]
  The error term $R_0^{(S)}$ is now a quadrature error term corresponding to a
  midpoint rule for the function $f$ in \eqref{eq:ffunc}, that can be written from
  Peano representation kernel as
  \[  \int_0^{\tau} f(s) \dd s -\tau f \left(\frac{\tau}{2} \right)= \int_0^{\tau} K_1(\theta) f''(\theta ) \dd \theta \]
  with
  \[ K_1(\theta)= \left\{
      \begin{aligned}
        & \frac12 \theta^2 \quad & \text{if} \ 0 \leq \theta \leq \frac{\tau}{2}, \\
        & \frac12 (\tau-\theta)^2 & \text{if} \ \frac{\tau}{2} \leq \theta \leq \tau.
      \end{aligned} \right. \]
  We now compute the second derivative of $f$ which writes
  \begin{align*}
    f''(s)& = D_A^2 \exp(sD_A) D_B\exp((\tau-s)D_A)\Id(v_0) - 2D_A \exp(sD_A)D_BD_A\exp((\tau-s)D_A)\Id(v_0) \\
    & \quad + \exp(sD_A) D_B D_A^2 \exp((\tau-s)D_A)\Id(v_0) \\
    & = \exp(sD_A) \left[ D_A,\left[D_A,D_B\right] \right]\exp((\tau-s)D_A)\Id(v_0)\\
    & = \exp(sD_A) D_{[[B,A],A]} \exp((\tau-s)D_A)\Id(v_0)\\
    & = e^{-i(\tau-s)\Delta} \left(\left[ A,\left[A,B\right]\right] (\Phi_A^s(v_0)) \right)
  \end{align*}
  since $ \left[ \left[B,A\right] ,A\right]=\left[ A,\left[A,B\right] \right] $
  and where we used again Lemma~\ref{lem:app0}.
  From Lemma~\ref{lemma_commutator_estimate} and
    \eqref{eq:prop_continuity_A} we then infer
  \[ \| R_0^{(S)} \|_{H^2} \leq C_0\tau^3  \]
  where  $C_0 = C(\| v_0\|_{H^6}$, $\| \phi \|_{X^6}, \| V
  \|_{H^6}) > 0$. To estimate
  $R_1-R_2^{(S)}$, we introduce the function
  \[  g(s,\sigma)=\exp(\sigma D_A) D_B \exp((s-\sigma)D_A) D_B \exp((\tau-s)D_A) \Id(v_0).  \]
  We show next that there is a compensation between the last two remaining error
  terms, writing that
  \begin{align*}
    & R_1-R_2^{(S)}  = \int_0^\tau \int_0^s g(s,\sigma) \dd \sigma \dd s - \frac{\tau^2}{2} g\left(\frac{\tau}{2},\frac{\tau}{2} \right) \\
    & \quad + \int_0^\tau \int_0^s  \left( \exp(\sigma D_H) - \exp(\sigma D_A) \right) D_B \exp((s-\sigma)D_A) D_B \exp((\tau-s)D_A) \Id(v_0) \dd \sigma \dd s \\
    & \quad +\tau^2 \left( \frac12 \exp \left(\frac{\tau}{2} D_A \right) D_B^2 \exp\left(\frac{\tau}{2} D_A \right) \Id(v_0) \right. \\
    & \quad \quad \quad \quad \left. - \int_0^1 (1-\theta) \exp\left(\frac{\tau}{2} D_A\right) \exp(\theta \tau D_B) D_B^2 \exp\left(\frac{\tau}{2} D_A \right) \Id(v_0) \dd \theta \right) \\
    & \quad \coloneqq R_3^{(S)}+R_4^{(S)}+R_5^{(S)}.
  \end{align*}

  As pointed out in \cite[Section~5]{Lubich2008}, $R_3^{(S)}$ is the quadrature
  error of a first-order two-dimensional quadrature formula, thus denoting the
  triangle
  \[ K_\tau=\enstq{(s,\sigma)\in \left[0,\tau \right]^2}{0 \leq s \leq \tau,\ 0 \leq \sigma \leq \tau-s},\]
  we infer the bound by multivariate Taylor expansion
  \[ \left\| \int_0^\tau \int_0^s g(s,\sigma) \dd \sigma \dd s -
      \frac{\tau^2}{2} g\left(\frac{\tau}{2}  ,\frac{\tau}{2} \right)
    \right\|_{H^2} \leq C \tau^3 \left( \sup_{K_\tau} \left\|\frac{\partial
          g}{\partial s}(s,\sigma)  \right\|_{H^2} + \sup_{K_\tau}
      \left\|\frac{\partial g}{\partial \sigma}(s,\sigma)   \right\|_{H^2}
    \right).   \]
  We readily compute, using Lemma~\ref{lem:app3},
  \begin{align*}
    \frac{\partial g}{\partial s}(s,\sigma)  =&  \exp(\sigma D_A) D_B
    \exp((s-\sigma)D_A) D_{[B,A]} \exp((\tau-s)D_A) \Id(v_0) \\
    = &  e^{-i(\tau-s)\Delta} \left(  [B,A]'(\Phi_A^s(v_0)) \left[
        e^{-i(s-\sigma) \Delta} B (\Phi_A^\sigma(v_0)) \right] \right)
  \end{align*}
  and
  \begin{align*} \frac{\partial g}{\partial \sigma}(s,\sigma) & = \exp(\sigma
    D_A) D_{[B,A]} \exp((s-\sigma)D_A) D_B \exp((\tau-s)D_A) \Id(v_0) \\
    & = e^{-i(\tau-s)\Delta} \left(  B'(\Phi_A^s(v_0)) \left[
        e^{-i(s-\sigma) \Delta} [B,A] (\Phi_A^\sigma(v_0)) \right]
    \right).
  \end{align*}
  Hence, from Lemmas \ref{lemma_nonlinear_operator_estimate} and
  \ref{lemma_commutator_estimate} together with
    \eqref{eq:prop_continuity_A}, we get
  \[ \| R_3^{(S)} \|_{H^2} \leq C_3 \tau^3,   \]
  where $C_3 = C(\| v_0 \|_{H^4}, \| \phi \|_{X^4}, \| V \|_{H^4})>0$. To bound $R_4^{(S)}$, we recall by Duhamel's formula that
  \[ \exp(\sigma D_H)\Id(\varphi) =  \exp(\sigma D_A) \Id(\varphi) + \int_0^\sigma
    \exp(\nu D_H)D_B\exp((\sigma-\nu)D_A) \Id(\varphi) \dd \nu \]
  hence, by Lemma~\ref{lem:app3},
    \[
      \begin{split}
        R_4^{(S)}=  & \int_0^\tau \int_0^s  \int_0^\sigma \exp(\nu
        D_H)D_B\exp((\sigma-\nu)D_A) D_B \exp((s-\sigma)D_A) D_B \exp((\tau-s)D_A)
        \Id(v_0) \dd \nu \dd \sigma \dd s \\
        = & \int_0^\tau \int_0^s \int_0^\sigma e^{-i(\tau-s)\Delta}\Big(
        B'\left(\Phi_A^{s-\nu}(v(\nu))\right)\left[e^{-i(s-\sigma)\Delta}B'(\Phi_A^{\sigma-\nu}(v(\nu)))
          [e^{-i(\sigma-\nu)\Delta}B(v(\nu))]\right] \\
        &\phantom{\int_0^\tau \int_0^s \int_0^\sigma e^{-i(\tau-s)\Delta} }
        + B''\left(\Phi_A^{s-\nu}(v(\nu))\right)
        \left[e^{-i(s-\nu)\Delta}B(v(\nu)),
          e^{-i(s-\sigma)\Delta}B(\Phi_A^{\sigma-\nu}(v(\nu)))\right]\Big)
        \dd \nu \dd \sigma \dd s
      \end{split}
    \]
  This leads, with Lemma~\ref{lemma_nonlinear_operator_estimate} and
    \eqref{eq:prop_continuity_A}, to
  \[ \| R_4^{(S)} \|_{H^2} \leq C_4 \tau^3,   \]
  where $C_4=C(\| v_0 \|_{H^2}, \| \phi \|_{X^2}, \| V \|_{H^2})>0$. Finally,
  to estimate $R_5^{(S)}$, by Taylor expansion we write
  \[  \exp(\theta \tau D_B) = \Id + \theta \tau \int_0^1 \exp ( \kappa \theta \tau D_B )D_B \dd \kappa, \]
  thus, using Lemma~\ref{lem:app4},
    \[
      \begin{split}
        R_5^{(S)} &= \tau^3 \int_0^1 \int_0^1 (1-\theta) \theta \exp\left(\frac{\tau}{2}
          D_A \right) \exp ( \kappa \theta \tau D_B )D_B^3 \exp\left(\frac{\tau}{2} D_A
        \right) \Id(v_0) \dd \kappa \dd \theta \\
        &=  \tau^3 \int_0^1 \int_0^1 (1-\theta) \theta
        e^{-i\frac\tau2\Delta}\Big(B''(\eta_{\kappa\theta\tau})
        \big[B(\eta_{\kappa\theta\tau}),B(\eta_{\kappa\theta\tau})\big] +
        B'(\eta_{\kappa\theta\tau})\big[B'(\eta_{\kappa\theta\tau})
        [B(\eta_{\kappa\theta\tau})]\big]\Big) \dd \kappa \dd \theta
      \end{split}
    \]
  where $\eta_{\kappa\tau\theta} =
  \Phi_B^{\kappa\tau\theta}(\Phi_A^{\frac\tau2}(v_0))$.
  Once again, using Lemma~\ref{lemma_nonlinear_operator_estimate} and
    \eqref{eq:prop_continuity_A} we get
  \[ \| R_5^{(S)} \|_{H^2} \leq C_5 \tau^3,   \]
  with $C_5 = C(\| v_0 \|_{H^2}, \| \phi \|_{X^2}, \| V \|_{H^2})>0$,
  which ends the proof.
\end{proof}

\subsection{Non-autonomous systems} \label{sec:non_autonomous}

We now detail how to extend the previous results to non-autonomous systems, that
is when the potential $V$ has an explicit dependency on time. As suggested by
\cite[Section~3.6]{blanesConciseIntroductionGeometric2025}
and \cite[Section~3.5]{blanesSplittingMethodsDifferential2024} in the finite
dimensional case, the main idea is to consider an extended variable
$\tv \coloneqq \binom{v}{t}\in H^1(\R^d)\times\R$ in order to rewrite
non-autonomous systems as autonomous (modified) systems, for which we can
perform the same local error study. Hereafter, we define $Y^k_T = H^k(\R^d)
\times [0,T]$, equipped with the norm $\|\tv\|_{Y^k}^2 \coloneqq \|v\|_{H^2}^2 +
|t|^2$, and quantities with a \emph{tilde} refer to \enquote{autonomized}
quantities, that is with the additional time component.

\medskip
\textbf{Modified autonomous system.} Let $v(t)\in H^k(\R^d)$ be the
solution to \eqref{eq:GP_v} with initial condition $v(t_0) = v_0$.
Then, the extended variable $\tv(t) = \binom{v(t)}{t} \in
H^k(\R^d)\times \R$ is solution to
\begin{equation}\label{eq:mod_GP_v}
  \begin{cases}
    \partial_t \tv(t) = \binom{\partial_t v(t)}{1} =
    \binom{A(v(t))}{0} + \binom{B(v(t),t)}{1} = \tA(\tv(t)) +
    \tB(\tv(t)), \\
    \tv(t_0) = \binom{v_0}{t_0},
  \end{cases}
\end{equation}
where
\[
  \tA(\tv) = \binom{A(v)}{0} \quad\text{and}\quad
  \tB(\tv) = \binom{B(v,t)}{1}.
\]
Recalling \eqref{eq:flows_splitting_v}, the associated flows read, for an
initial state $\tv_0 = \binom{v_0}{t_0}$,
\[
  \tPhi_{\tA}^t(\tv_0) = \binom{\Phi_A^t(v_0)}{t_0} \quad\text{and}\quad
  \tPhi_{\tB}^t(\tv_0) = \binom{\Phi_B^{t,t_0}(v_0)}{t_0+t}.
\]
One can easily check that
\[
  \begin{cases}
    \partial_t \tPhi_{\tA}^t(\tv) = \tA(\tPhi_{\tA}^t(\tv)),\\
    \tPhi_{\tA}^0(\tv) = \tv_0,
  \end{cases}\quad\text{and}\quad
  \begin{cases}
    \partial_t \tPhi_{\tB}^{t}(\tv) = \tB(\tPhi_{\tB}^t(\tv)),\\
    \tPhi_{\tB}^0(\tv) = \tv_0.
  \end{cases}
\]
Moreover, from \eqref{eq:prop_continuity_A} and
Proposition~\ref{prop_continuity}, we have for any $\tvphi =
\binom{\varphi}{t_0}\in Y^k_T$,
\[
  \|\tPhi_{\tA}^t(\tvphi) \|_{Y^k} \leq \|\varphi\|_{H^k} + \sqrt{t}\|\phi\|_{X^k} +
  |t_0|
\]
and
\[
  \|\tPhi_{\tB}^t(\tvphi) \|_{Y^k} \leq \| \Phi_B^{t,t_0}(\varphi) \|_{H^2} + |t_0+t|
  \leq \exp(Ct) + |t_0+t|,
\]
where $C>0$ depends on $\| \varphi\|_{H^k}$, $\| \phi\|_{H^{k}}$ and $\| V \|_{\mathcal{C}^0_T H^k}$.

The Lie and Strang splitting schemes to go from $\tv^n =
\binom{v^n}{t^n}$ to $\tv^{n+1} = \binom{v^{n+1}}{t^{n+1}}$ then
write as follows.
\begin{description}
  \item[Lie] We compute
    \[
      \begin{split}
        \tv^{n+1}_L &= \tPhi_{\tB}^\tau \circ \tPhi_{\tA}^\tau (\tv^n)
        = \tPhi_{\tB}^\tau \binom{\Phi_A^\tau(v^n)}{t^n}
        = \binom{\Phi_B^{\tau,t^n}(\Phi_A^\tau(v^n))}{t^n + \tau}
        = \binom{v^{n+1}_L}{t^{n+1}},
      \end{split}
    \]
    where $v^{n+1}_L = \Phi_B^{\tau,t^n}(\Phi_A^\tau(v^n))$ is the
    (non-autonomous) Lie splitting without the additional time component.
  \item[Strang] We compute
    \[
      \begin{split}
        \tv^{n+1}_S &= \tPhi_{\tA}^{\tau/2} \circ \tPhi_{\tB}^{\tau}
        \circ \tPhi_{\tA}^{\tau/2} (\tv^n)
        = \tPhi_{\tA}^{\tau/2} \circ \tPhi_{\tB}^\tau
        \binom{\Phi_A^{\tau/2}(v^n)}{t^n} \\
        & = \tPhi_{\tA}^{\tau/2}
        \binom{\Phi_B^{\tau,t^n}(\Phi_A^{\tau/2}(v^n))}{t^n + \tau}
        = \binom{\Phi_A^{\tau/2}
          (\Phi_B^{\tau, t^n}(\Phi_A^{\tau/2}(v^n)))}
        {t^n + \tau}
        = \binom{v^{n+1}_S}{t^{n+1}}
      \end{split}
    \]
    where $v^{n+1}_S = \Phi_A^{\tau/2}(\Phi_B^{\tau,
      t^n}(\Phi_A^{\tau/2}(v^n))$ is the (non-autonomous) Strang splitting without the
    additional time component.
\end{description}
As a conclusion, (i) the PDE \eqref{eq:mod_GP_v} is an \emph{autonomous} PDE and
(ii) the Lie and Strang splitting schemes on the associated operators $\tA$ and
$\tB$ are equivalent to the same splitting schemes for the non-autonomous
version. One can thus apply the exact same strategy as the one we used in the
autonomous case in order to derive convergence rates for both schemes, provided
that the error bounds and commutator bounds from
Lemmas~\ref{lemma_nonlinear_operator_estimate}-\ref{lemma_commutator_estimate}
are adapted accordingly.

\medskip
\textbf{Derivatives and bounds for the modified system} The Fréchet derivatives
of $\tA$ and $\tB$ can be computed as follows, for any
$\tvphi = \binom{\varphi}{t}$ and $\tpsi = \binom{\psi}{s}$:
\begin{equation}\label{eq:frechet_modified_flow}
  \tA'(\tvphi)[\tpsi] = \binom{A'(\varphi)[\psi]}{0} \quad\text{and}\quad
  \tB'(\tvphi)[\tpsi] = \binom{\partial_\varphi B(\varphi,t)[\psi] + \partial_t
    B(\varphi,t)s}{0},
\end{equation}
where $\partial_\varphi B(\varphi,t)$ is the Fréchet derivative of $\varphi
\mapsto B(\varphi,t)$, which coincides with the Fréchet derivative of $B$ in the
autonomous case with fixed $t$. Note also that
\[
  \partial_tB(\varphi,t) = -i \partial_tV(t)(\phi+\varphi).
\]

\begin{lemma} \label{lemma_nonlinear_operator_estimate_nonautonom}
  For any $\tvphi$, $\tpsi$, $\tchi \in Y^2_T(\R^d)$, we
  have
  \[ \begin{split}
      \| \tB (\tvphi) \|_{Y^2} & \leq C_d \left( \| \phi\|_{X^2} +
        \|\varphi\|_{H^2} \right) \left(1 + \| V\|_{\mathcal{C}^0_T H^2} + \| \phi\|_{X^2}^2 + \|\varphi\|_{H^2}^2 \right)
      +1\\
      \| \tB' (\tvphi) [\tpsi] \|_{Y^2} & \leq
      C_d \| \psi\|_{H^2} \left(1 + \| V\|_{\mathcal{C}^0_T H^2} + \| \phi\|_{X^2}^2 + \|\varphi\|_{H^2}^2 \right)
      + C_d\| V\|_{{\mathcal{C}^1_T H^2}}(\| \phi \|_{X^2} + \|\varphi\|_{H^2})T \\
      \| \tB'' (\tvphi) [\tpsi,\tchi] \|_{Y^2}
      & \leq C_d \| \psi\|_{H^2} \| \chi\|_{H^2} \left( \| \phi\|_{X^2} +
          \|\varphi\|_{H^2} \right)
        + C_d \| V \|_{C^1_T H^2} (\| \psi\|_{X^2} + \|\chi\|_{H^2} )T  \\
      & \quad + C_d \| V \|_{C^2_T H^2} (\| \phi\|_{X^2} + \|\varphi\|_{H^2} )T^2.
    \end{split} \]
\end{lemma}

\begin{proof}
  The first bound simply follows from
  \[
    \|\tB(\tvphi)\|_{Y^2} \leq \|B(\varphi,t)\|_{H^2} + 1
  \]
  and the use of the first bound in
  Lemma~\ref{lemma_nonlinear_operator_estimate}, with $\| V\|_{\mathcal{C}^0_T H^2}$
  to bound the potential term.
  As for the second bound, we directly get from \eqref{eq:frechet_modified_flow}
  and $\tpsi = \binom{\psi}{s}$
  \[
    \| \tB' (\tvphi) [\tpsi] \|_{Y^2} = \| \partial_\varphi B(\varphi,t)[\psi] +
    \partial_t B(\varphi,t)s \|_{H^2} \leq
    \| \partial_\varphi B(\varphi,t)[\psi] \|_{H^2} + \|\partial_t B(\varphi,t)
    \|_{H^2}|s|
  \]
  and the result follows as before from the second bound of
  Lemma~\ref{lemma_nonlinear_operator_estimate} together with
  \[
    \|\partial_t B(\varphi,t)\|_{H^2}|s| \leq C_d\|\partial_t V\|_{\mathcal{C}^0_TH^2}
    (\| \phi \|_{X^2} + \|\varphi\|_{H^2})T.
  \]
  The bound on $\tB''$ similarly follows noticing that, for
  $\tpsi = \binom{\psi}{s}$ and $\tchi = \binom{\chi}{r}$,
    \[  \begin{split}
    \tB'' (\tvphi) [\tpsi,\tchi]  & = \binom{
      \partial_\varphi^2 B(\varphi,t)[\psi,\chi]
      + \partial_\varphi(\partial_tB(\varphi,t)){[r\psi+s\chi]}
      + \partial_t^2 B(\varphi,t) sr}{0} \\
        & = \binom{\partial_\varphi^2 B(\varphi,t)[\psi,\chi]
        - i \partial_t V(t){(r\psi+s\chi)}
        -i \partial_t^2 V(t)(\phi +\varphi) sr}{0}
    \end{split}  \]
  and taking the $H^2$ norm of the first component together with the third bound
  of Lemma~\ref{lemma_nonlinear_operator_estimate}.
\end{proof}

\begin{lemma} \label{lemma_commutator_estimate_nonautonom}
  For any $\tvphi$, $\tpsi \in Y^4_T(\R^d)$, we have
  \[ \begin{split}
      \| [ \tA,\tB ] (\tvphi) \|_{ Y^2} & \leq
      C_d \left( \| V\|_{\mathcal{C}^0_T H^4}^2 + \| \phi\|_{X^4}^3+\|\varphi \|_{H^4}^3  \right), \\
      \| [ \tA,\tB ]' (\tvphi) [\tpsi] \|_{Y^2} & \leq
      C_d \|\psi\|_{H^4} \left(\| V\|_{\mathcal{C}^0_TH^4}^2
        + \| \phi\|_{X^4}^2+\|\varphi \|_{H^4}^2 \right)
        + C_d \|V\|_{\mathcal{C}^1_TH^4}(\|\phi\|_{X^3} + \|\varphi\|_{H^3})T,
    \end{split} \]
  and for any $\tvphi \in Y^6_T$,
  \[\| [\tA, [\tA,\tB] ]  (\tvphi) \|_{ Y^2} \leq  C_d \left(\| V\|_{\mathcal{C}_T^0 H^6}^2 + \| \phi\|_{X^6}^3+\|\varphi \|_{H^6}^3  \right) .  \]
\end{lemma}

\begin{proof}
  The first commutator bound follows after noting that
  \[
    [\tA,\tB](\tvphi) = \tA'(\tvphi)[\tB(\tvphi)] - \tB'(\tvphi)[\tA(\tvphi)]
    = \binom{A'(\varphi)[B(\varphi,t)] -
      \partial_{\varphi}B(\varphi,t)[A(\varphi)]}{0} =
    \binom{[A,B(\cdot,t)](\varphi)}{0}
  \]
  and using the first bound from Lemma~\ref{lemma_commutator_estimate}. From
  this last expression we directly infer that, for $\tpsi = \binom{\psi}{s}$,
  \[
    \|[ \tA,\tB ]' (\tvphi) [\tpsi]\|_{Y^2}
    = \|\partial_\varphi ([A,B(\cdot,t)])(\varphi)[\psi] +
    \partial_t ([A,B(\cdot,t)])(\varphi)s\|_{H^2}
  \]
  Noting that the Fréchet derivative of $\varphi \mapsto -i\partial_t
  V(t)(\phi+\varphi)$ is just the multiplication by $-i\partial_t V(t)$, we
  compute
  \[
    \begin{split}
      \partial_t ([A,B(\cdot,t)])(\varphi) &= [A,\partial_t B(\cdot,t)](\varphi)
      = [A,-i\partial_tV(t)(\phi+\cdot)](\varphi) \\
      &= A'(\varphi)[-i\partial_tV(t){(\phi+\varphi)}] -
      (-i\partial_tV(t))A(\varphi) \\
      &= -\partial_t(\Delta V(t))(\phi+\varphi) - 2\partial_t(\nabla V(t))\cdot
      (\nabla \phi + \nabla \varphi),
    \end{split}
  \]
  we obtain, with the second bound of Lemma~\ref{lemma_commutator_estimate},
  \[ \begin{split}
      \|[ \tA,\tB ]' (\tvphi) [\tpsi]\|_{Y^2} &\leq
      \|\partial_\varphi ([A,B(\cdot,t)])(\varphi)[\psi] \|_{H^2} +
      \|\partial_t ([A,B(\cdot,t)])(\varphi)s\|_{H^2} \\
      & \leq C_d \|\psi\|_{H^4} \left(\| V\|_{\mathcal{C}^0_TH^4}^2
        + \| \phi\|_{X^4}^2+\|\varphi \|_{H^4}^2 \right)
        + C_d \|V\|_{\mathcal{C}^1_TH^4}(\|\phi\|_{X^3} + \|\varphi\|_{H^3})T,
    \end{split}
  \]
  which yields the second bound. Finally, from \eqref{eq:frechet_modified_flow}, we compute
  \[ [\tA, [\tA,\tB]
    ]  (\tvphi)= \binom{[A, [A,B(\cdot,t)] ](\varphi)}{0}, \]
  from which we obtain the last bound using once again the third bound from
  Lemma~\ref{lemma_commutator_estimate}.
\end{proof}

We now move to the study of the local error for both Lie and Strang splitting
schemes.

\medskip
\textbf{(Lie splitting)}
Let $\tv_0 = \binom{v_0}{t_0} \in Y^4_T$. Since $\| \tv(\tau) - \tv_L^{1}
\|_{Y^2_T} = \| v(\tau) - v_L^{1} \|_{H^2}$, it remains to show, as in the
autonomous case, that
\[
  \| \tv(\tau) - \tv_L^{1} \|_{Y^2_T} \leq \tC_L \tau^2,
\]
where $\tC_L=C(\|v_0\|_{H^4}, \| V\|_{\mathcal{C}^0_TH^4}, \| \phi \|_{X^4},T) >0$. Following the proof of Proposition~\ref{prop:local_error}, it holds
\[
  \tv(\tau) - \tv_L^1 = \tR_0^{(L)} + \tR_1 - \tR_2^{(L)}
\]
where
\[
  \begin{split}
    \tR_0^{(L)} &= \int_0^{\tau} \exp(s D_{\tA}) D_{\tB} \exp((\tau-s)D_{\tA})
    \Id(\tv_0) \dd s - \tau \exp(\tau D_{\tA}) D_{\tB} \Id(\tv_0), \\
    \tR_1       &= \int_0^{\tau} \int_0^s \exp(\sigma D_{\tH}) D_{\tB}
    \exp((s-\sigma)D_{\tA} ) D_{\tB} \exp ( (\tau-s)D_{\tA}) \Id(\tv_0) \dd
    \sigma \dd s,\\
    \tR_2^{(L)} &= \tau^2 \int_0^1  (1-\theta) \exp(\tau D_{\tA}) \exp(\theta \tau
    D_{\tB})D_{\tB}^2 \Id(\tv_0) \dd \theta.
  \end{split}
\]
Hence, using \eqref{eq:R0_Lie}-\eqref{eq:R1_Lie}-\eqref{eq:R2_Lie} with
Lemmas~\ref{lemma_nonlinear_operator_estimate_nonautonom} and
\ref{lemma_commutator_estimate_nonautonom} to bound the integrands in each of
these terms, we get
\[
  \begin{split}
    \|\tR_0^{(L)}\|_{Y^2_T} &\leq \tC_0 \tau^2 \quad\text{with}\quad
    \tC_0 = C(\|v_0\|_{H^4}, \|\phi\|_{X^4}, \| V\|_{\mathcal{C}^0_TH^4}, T)>0, \\
    \|\tR_1\|_{Y^2_T}       &\leq \tC_1 \tau^2 \quad\text{with}\quad
    \tC_1 = C(\|v_0\|_{H^2}, \| \phi \|_{X^2}, \|V\|_{\mathcal{C}^1_TH^2}, T)>0, \\
    \|\tR_2^{(L)}\|_{Y^2_T} &\leq \tC_2 \tau^2 \quad\text{with}\quad
    \tC_2 = C(\|v_0\|_{H^2}, \| \phi \|_{X^2}, \|V\|_{\mathcal{C}^1_TH^2}, T)>0.
  \end{split}
\]

\medskip
\textbf{(Strang splitting)} We now assume that $\tv_0 = \binom{v_0}{t_0} \in Y^6_T$, and we write that
  \[ v(\tau)-v_s = \tR_0^{(S)} + \tilde{R}_3^{(S)} + \tilde{R}_4^{(S)} + \tilde{R}_5^{(S)}, \]
  where $\tR_0^{(S)}$, $\tilde{R}_3^{(S)}$, $\tilde{R}_4^{(S)}$ and $\tilde{R}_5^{(S)}$ are defined as in the proof of Proposition~\ref{prop:local_error}
  with $\tA$ and $\tB$ instead of $A$ and $B$. We then infer from Lemmas~\ref{lemma_nonlinear_operator_estimate_nonautonom} and
  \ref{lemma_commutator_estimate_nonautonom} that
  \[
    \begin{split}
      \|\tR_0^{(S)}\|_{Y^2_T} &\leq \tC_0 \tau^3 \quad\text{with}\quad
      \tC_0 = C(\|v_0\|_{H^6}, \|\phi\|_{X^6}, \| V\|_{\mathcal{C}^0_TH^6}, T)>0, \\
      \|\tR_3^{(S)}\|_{Y^2_T} &\leq \tC_3 \tau^3 \quad\text{with}\quad
      \tC_3 = C(\|v_0\|_{H^4}, \| \phi \|_{X^4}, \|V\|_{\mathcal{C}^1_TH^4}, T)>0, \\
      \|\tR_4^{(S)}\|_{Y^2_T} &\leq \tC_4 \tau^3 \quad\text{with}\quad
      \tC_4 = C(\|v_0\|_{H^2}, \| \phi \|_{X^2}, \|V\|_{\mathcal{C}^2_TH^2}, T)>0, \\
      \|\tR_5^{(S)}\|_{Y^2_T} &\leq \tC_5 \tau^3 \quad\text{with}\quad
      \tC_5 = C(\|v_0\|_{H^2}, \| \phi \|_{X^2}, \|V\|_{\mathcal{C}^2_TH^2}, T)>0.
    \end{split}
  \]
\medskip
To wrap things up, projecting on the first component of $Y^2_T$ we just showed the non-autonomous version of
Proposition~\ref{prop:local_error}:
\begin{proposition}[Non-autonomous case] \label{prop:local_error_nonautonom}
  We have

    \medskip
    \textbf{(Lie splitting)} $\| v(\tau) - v_L^1 \|_{H^2} \leq C_L\tau^2$,
    where $C_L=C(\|v_0\|_{H^4}, \| V\|_{\mathcal{C}^0_TH^4},\| V\|_{\mathcal{C}^1_TH^2}, \| \phi \|_{X^4}, T),$

    \medskip
        \textbf{(Strang splitting)}  $\| v(\tau) - v_S^1 \|_{H^2} \leq {C}_S \tau^3$,     where ${C}_S={C}(\|v_0\|_{H^6}, \| V\|_{\mathcal{C}^0_T H^6},
       \| V\|_{\mathcal{C}^1_T H^4},
    \| V\|_{\mathcal{C}^2_T H^2}, \| \phi \|_{X^6}, T)$.

\end{proposition}

\subsection{Convergence estimates}  \label{sec:convergence}
We now have all the elements to turn to the proof of our main result.

\begin{proofof}{Proposition}{\ref{prop:convergence_splitting_v}}

  \medskip
  \textbf{(Lie splitting)} We prove by induction that $(v_L^n)_{n\in\N}$ is
  uniformly bounded with respect to $\tau$ in $H^2(\R^d)$, which will induce the
  convergence result. More precisely, we show that if we denote
  $m_4=\|v\|_{\mathcal{C}^0_T H^4}$, then for all $ 0 \leq n \leq N$, we have
  $v^n \in B_{m_4+1}$, where $B_{m_4+1}$ denotes the centered ball of radius
  $1+m_4$ in $H^2(\R^d)$. The initialization is direct as $v^0_L=v(0)$. Let us
  now assume that $v^k_L \in B_{m_4+1}$ for all $0\leq k \leq n$. Using
  respectively the local error estimate from Proposition~\ref{prop:local_error}
  in the autonomous case or Proposition~\ref{prop:local_error_nonautonom} in
    the non-autonomous case, together with the stability estimate from Proposition
  \ref{prop:stability} with $M_2 = 1+m_4$, we write that
  \begin{align*}
    \| v(t_{n+1})-v^{n+1}_L \|_{H^2} & \leq \|v(t_{n+1}) -
    \Phi_L^{\tau,t_n}(v(t_n)) \|_{H^2} + \| \Phi_L^{\tau,t_n}(v(t_n)) -
    \Phi_L^{\tau,t_n}(v^n_L) \|_{H^2} \\
    & \leq C(m_4) \tau^2 + e^{\tau C(M_2)} \| v(t_n) -  v^n_L \|_{H^2} \\
    & \leq C(m_4) \tau^2 \sum_{k=0}^n e^{\tau C(M_2)k}
  \end{align*}
  where the last inequality is obtained by a recursive argument. Note that $C(m_4)$ and $C(M_2)$ are constants which depend respectively on $m_4$ and $M_2$ among other parameters, but which are uniform with respect to $\tau$ and $n$. As a sum of terms
  of a geometric sequence and as $(n+1)\tau \leq  T$, this gives, using
    $1+\tau C(M_2) \leq e^{\tau C(M_2)}$,
  \begin{equation} \label{eq_recursive_convergence}
    \|v(t_{n+1})-v^{n+1}_L \|_{H^2}
    \leq C(m_4) \tau^2 \frac{e^{C(M_2)T}-1}{e^{\tau C(M_2)} - 1}
    \leq C(m_4) \tau \frac{e^{C(M_2)T}-1}{C(M_2)}.
  \end{equation}
   Hence, for $\tau \leq \tau_L \coloneqq \frac{C(M_2)}{C(m_4)(e^{C(M_2)T}-1)}$ we get
  \[ \| v^{n+1}_L \|_{H^2} \leq \| v(t_{n+1})\|_{H^2}+\|v(t_{n+1})-v^{n+1}_L
    \|_{H^2}  \leq m_4+1.\]
  This yields $v^{n+1} \in B_{m_4+1}$ which ends the induction proof. The convergence
  result is a direct consequence of equation \eqref{eq_recursive_convergence}.

  \medskip
  \textbf{(Strang splitting)} Assuming $v^n \in B_{m_6+1}$ with
  $m_6=\|v\|_{\mathcal{C}^0_T H^6}$, the exact same way we infer that
  \begin{align*}
    \| v(t_{n+1})-v^{n+1}_S \|_{H^2} & \leq \|v(t_{n+1}) -
    \Phi_S^{\tau,t_n}(v(t_n)) \|_{H^2} + \| \Phi_S^{\tau,t_n}(v(t_n)) -
    \Phi_S^{\tau,t_n}(v^n_S) \|_{H^2} \\
    & \leq C(m_6) \tau^3 + e^{\tau C(M_2)} \| v(t_n) -  v^n_S \|_{H^2} \\
    & \leq C(m_6) \tau^2 \frac{e^{C(M_2)T}-1}{C(M_2)}.
  \end{align*}
  Similarly, for $\tau^2 \leq {\tau}_S^2 \coloneqq
  \frac{C(M_2)}{C(m_6)(e^{C(M_2)T}-1)}$ we
  get $v^{n+1}_S \in B_{m_6+1}$, which concludes the bootstrap and
  gives the appropriate convergence rate.
\end{proofof}

\bigskip
We are now able to conclude the proof of our main result.

\begin{proofof}{Theorem}{\ref{theorem_convergence}}
  With the help of Lemma~\ref{lem:equivalence_splittings} one recursively computes
    for the Lie splitting, from $u_0 = \phi + v_0$,
    \[
      \begin{split}
        u_\mathcal{L}^{n+1} &= \Phi_\mathcal{L}^{\tau,t_n}(u_\mathcal{L}^n)
        = \Phi_\mathcal{B}^{\tau,t_n} \circ \Phi_\mathcal{A}^\tau (u^n_\mathcal{L})
        = \Phi_\mathcal{B}^{\tau,t_n} \circ \Phi_\mathcal{A}^\tau (\phi + v_L^n)
        = \Phi_\mathcal{B}(\phi + \Phi_A^\tau(v_L^n)) \\
        &= \phi + \Phi_B^{\tau,t_n}\circ\Phi_A^\tau (v_L^n) = \phi +
        \Phi_L^{\tau,t_n}(v_L^n) = \phi + v_L^{n+1}.
      \end{split}
    \]
    Therefore, we simply write for the Lie scheme that
  \[ \| u(t_n)-u^n_{\mathcal{L}} \|_{X^2} =  \| (\phi+ v(t_n))-(\phi + v^n_L)
    \|_{X^2} \leq \| v(t_n)- v^n_L \|_{H^2} \]
  as $H^2(\R^d) \subset L^{\infty}(\R^d)$ for $1 \leq d \leq 3$ . The result follows from
  Proposition \ref{prop:convergence_splitting_v}. The same holds for the Strang
  splitting.
\end{proofof}

\begin{remark} \label{rem:influence_eps}
   We now take $\eps>0$ instead of $\eps=1$ in \eqref{eq:GP}. Following the
     computations of Section~\ref{sec:stability} and \ref{sec:local_error}, the
     stability estimates of Proposition~\ref{prop:stability} then writes
    \[ \| \Phi_B^{\tau,t_0}(f) - \Phi_B^{\tau,t_0}(g) \|_{H^2} \leq e^{\tau C(M_2)/\eps^2} \| f - g \|_{H^2},  \]
    while all estimates in Lemmas \ref{lemma_nonlinear_operator_estimate} and
    \ref{lemma_commutator_estimate} are multiplied by $1/\eps^2$ on their right
    hand side. Mimicking the proof of the convergence estimates as above yields
    \begin{equation} \label{eq:bounds_epsilon}
      \| v(t_{n}) - v^{n} \|_{H^2} \leq C \tau^\alpha e^{\tilde{C} T/\eps^2}
    \end{equation}
    for some constants $C$, $\tilde{C}>0$ independent of $\eps$, where $\alpha=1$ for
    the Lie scheme \eqref{eq:Lie_split_u} (that is $v^n=v^n_L$ in \eqref{eq:bounds_epsilon})
    and $\alpha=2$ for the Strang scheme \eqref{eq:Strang_split_u} (that is $v^n=v^n_S$
    in \eqref{eq:bounds_epsilon}). In particular, we observe that the bound
    \eqref{eq:bounds_epsilon} diverges exponentially fast in the singular limit
    $\eps \to 0$.
\end{remark}

\section{Mass and energy} \label{sec:consered_quantities}

\subsection{Preservation of the generalized mass}
One of the very interesting feature of splitting schemes for nonlinear
Schrodinger-type equations is that they inherently preserve the mass (namely the
number of particles of the physical system), which is usually the $L^2$-norm of
the solution. Of course, such quantity makes no sense in the case of equation
\eqref{eq:GP}, as for $u=\phi+v$ the quantity
\[  \mathcal{M}(u)=\int_{\R^d}(1-|u|^2)  = \int_{\R^d} (1-|\phi|^2
  -2\Re(\overline{\phi} v)-|v|^2)  \]
may not be defined for $v \in H^k(\R^d)$ with any $k \in \N$ and $\phi$
satisfying \eqref{eq:requirements_phi}.

However, one can define a notion of \textit{generalized mass} as follows (see
also \cite{BethuelGravejatSautSmets2010,DeLaire2010}). Let $\eta \in
\mathcal{C}_0^{\infty}(\R)$ such that $|\eta(x)| \leq 1$ for $|x| \leq 1$ and
$\eta(x)=0$ for $|x| \geq 2$, with $\| \eta'\|_{L^{\infty}},\|
\eta''\|_{L^{\infty}} \leq 2$, and define for any $R>0$ and $x_0 \in \R^d$ the
function
\[ \eta_{x_0,R}(x) = \eta \left( \frac{|x-x_0|}{R}   \right)   \]
for $x \in \R^d$. One can then define the  quantities
\[ m^+(u)= \inf_{x_0 \in \R^d} \limsup_{R \to \infty} \int_{\R^d} (1-|u|^2)
  \eta_{x_0,R} \quad \text{and} \quad m^-(u)= \inf_{x_0 \in \R^d} \liminf_{R \to
    \infty} \int_{\R^d} (1-|u|^2) \eta_{x_0,R}. \]
Then if $1-|u|^2 \in L^1(\R^d)$, we have $m^+(u)=m^-(u)$ and one can define the
generalized mass $\mathcal{M}(u) \coloneqq m^+(u)=m^-(u)$. It is then well-known (see e.g.
\cite[Theorem 7.7]{DeLaire2010}) that if $u_0$ has finite conserved generalized mass, then
for all $t\in \R$, $u(t)$ has finite generalized mass. We prove the following
conservation of the generalized mass result for our splitting scheme.

\begin{proposition}\label{prop:mass_preserv}
  With the same assumptions as in Theorem \ref{theorem_convergence} with $\tau
  \leq \tau_{\mathcal{L}}$ or $\tau_{\mathcal{S}}$, and the additional assumption that $u_0$ has a finite
  generalized mass $\mathcal{M}(u_0)$, we have $\mathcal{M}(u^n)=\mathcal{M}(u_0)$ for all $0 \leq n \leq N$,
  for either the Lie-Trotter splitting scheme $u^n=u^n_{\mathcal{L}}$
  \eqref{eq:Lie_split_u} or the Strang splitting scheme $u^n=u^n_{\mathcal{S}}$
\eqref{eq:Strang_split_u}.
\end{proposition}
\begin{proof}
  Note that since $| \Phi_{\mathcal{B}}^t(\zeta)|=|\zeta|$ for all $t \in \R$,
  it suffices to show the result for $\Phi_{\mathcal{A}}^t(\xi)$, hence for a
  solution to the linear equation $i \partial_t u =\Delta u$ with initial
  condition $u(0)=\xi$. As usual, we write $u=\phi+w$ with $w(0)=\xi-\phi \in
  H^2(\R^d)$. In particular $w$ is solution to the affine Schrödinger equation
  $i\partial_t w = \Delta w + \Delta \phi$, and thanks to
  \eqref{eq:continuity_linear_flow_Zhidkov} we have
  \begin{equation} \label{eq:bounds_linear_eq}
    \| w(t) \|_{L^2} \leq \| w_0 \|_{L^2}+ C \sqrt{t} \| \nabla \phi \|_{L^2}
    \quad \text{and} \quad   \| \nabla w(t) \|_{L^2} \leq \| \nabla w_0
    \|_{L^2}+ C \sqrt{t} \| \Delta \phi \|_{L^2}
  \end{equation}
  for all $t \geq 0$. We compute by integration by parts that
  \begin{multline*}
    \frac{\dd}{\dd t} \int_{\R^d}(1-|u(t)|^2) \eta_{x_0,R}  = -2 \Re \int_{\R^d}
    \overline{u(t)} \partial_t u(t) \eta_{x_0,R} \\
    = 2 \Im \int_{\R^d} \overline{w(t)}  (\nabla w(t)+\nabla \phi)  \nabla
    \eta_{x_0,R}   + 2 \Im \int_{\R^d}  \overline{\phi} \nabla w(t)  \nabla
    \eta_{x_0,R}+ 2 \Im \int_{\R^d}  \overline{\phi} \nabla \phi  \nabla
    \eta_{x_0,R}\\
    \eqqcolon J_1(t) + J_2(t)+ J_3(t).
  \end{multline*}
  We then notice that $\| \Delta \eta_{x_0,R} \|_{L^2} \leq C
  R^{\frac{d-4}{2}}$, implying that $\| \nabla \eta_{x_0,R} \|_{L^{\infty}}$
  and $\| \Delta \eta_{x_0,R} \|_{L^2}$ are uniformly bounded in $x_0$ and $R$.
  Writing
  \[ \Omega_{x_0,R} \coloneqq \enstq{x \in \R^d}{R<|x-x_0|<2R} , \]
  we infer from Cauchy-Schwarz inequality that
  \[  \left| J_1(t) \right|  \leq  C_\phi \| w(t) \|_{L^2(\Omega_{x_0,R})}
    \left(1+ \| \nabla w(t) \|_{L^2} \right). \]
  By integration by parts we infer
  \[  J_2(t)=2 \Im \int_{\R^d} w(t) \left( \nabla \overline{\phi} \nabla \eta_{x_0,R} + \overline{\phi} \Delta \eta_{x_0,R} \right)  \]
  which leads to
  \[  \left| J_2(t) \right|  \leq C_\phi \| w(t) \|_{L^2(\Omega_{x_0,R})}. \]
  Finally, noticing that the choice of $\eta$ implies that the quantity $\|
  \nabla \eta_{x_0,R} \|_{L^d}$ is uniformly bounded in $x_0$ and $R$ in any dimension,
  as well as $\| \nabla \eta_{x_0,R} \|_{L^2}$ in dimension one, Hölder
  inequality implies that
  \[ \left| J_3(t) \right| \leq \left\{  \begin{aligned}
        & \| \phi \|_{L^{\infty}} \| \nabla \phi \|_{L^2(\Omega_{x_0,R})} \| \nabla \eta_{x_0,R} \|_{L^2} & \quad \text{for } d=1, \\
        & \| \phi \|_{L^{\infty}} \| \nabla \phi \|_{L^{d/(d-1)}(\Omega_{x_0,R})} \| \nabla \eta_{x_0,R} \|_{L^d} & \quad \text{for } 2\leq d \leq 3.
      \end{aligned} \right. \]
  Gathering these inequalities and integrating in time we then have
  \begin{multline} \label{eq:bound_generalized_mass}
    \left| \int_{\R^d}(1-|u(t)|^2) \eta_{x_0,R} - \int_{\R^d}(1-|u(0)|^2) \eta_{x_0,R} \right| \\
    \leq C_\phi t \| \nabla \phi \|_{L^{d^*}(\Omega_{x_0,R})} + C_\phi \int_0^t  \| w(s) \|_{L^2(\Omega_{x_0,R})} \left( 1+ \| \nabla w(s) \|_{L^2} \right) \dd s
  \end{multline}
  with the convention $d^*=2$ if $d=1$ and $d^*=d/(d-1)$ if $2 \leq d \leq 3$. From \eqref{eq:bounds_linear_eq} we then infer by Cauchy-Schwarz inequality that
  \[   \int_0^t  \| w(s) \|_{L^2(\Omega_{x_0,R})} \left( 1+ \| \nabla w(s) \|_{L^2} \right) \dd s  \leq C_\phi (1+t) \left( \int_0^t \int_{\Omega_{x_0,R}} | w(s,x) |^2 \dd x\dd s \right)^{1/2}.  \]
  We conclude from the dominated convergence theorem in \eqref{eq:bound_generalized_mass} that
  \[ \lim_{R \to \infty} \left( \int_{\R^d}(1-|u(t)|^2) \eta_{x_0,R} - \int_{\R^d}(1-|u(0)|^2) \eta_{x_0,R}  \right) = 0  \]
  and the result follows from the definitions of $m^+(u_0)$, $m^-(u_0)$ and $\mathcal{M}(u_0)$.
\end{proof}

\subsection{Near-conservation of energy}
In the case of a time-independent potential (namely $\partial_t V=0$), the Ginzburg-Landau energy \eqref{eq:GL} is a constant of motion under the flow of \eqref{eq:GP}. We show that this property is nearly conserved by the splitting schemes \eqref{eq:Lie_split_u} and \eqref{eq:Strang_split_u}.

\begin{proposition}\label{prop:nrj_preserv}
  Under the same assumptions as in Theorem \ref{theorem_convergence} with $\tau
  \leq \tau_0$, and denoting $m_2,M_2>0$ two constants such that
  $\sup_{0 \leq t \leq T}  \|u(t)\|_{X^2}\leq m_2$ and
  $\sup_{0 \leq n \leq N} \|u^n\|_{X^2} \leq M_2$, we have
  \[
    |\mathcal{E}(u(t_n)) - \mathcal{E}(u^n)| \leq C(m_2,M_2) \tau^\alpha,
  \]
  for all $0 \leq n \leq N$, with $\alpha=1$ for the Lie-Trotter splitting scheme $u^n=u^n_{\mathcal{L}}$ \eqref{eq:Lie_split_u} and with $\alpha=2$ for the Strang splitting scheme $u^n=u^n_{\mathcal{S}}$ \eqref{eq:Strang_split_u}.
\end{proposition}
\begin{proof}
  We compare $\mathcal{E}(u(t_n))$ and $\mathcal{E}(u^n)$ term by term. For the kinetic part, we simply write that
  \begin{align*}
    \left| \int_{\R^d} | \nabla u(t_n)|^2 - \int_{\R^d} | \nabla u^n|^2   \right|  &= \left|\int_{\R^d} \nabla u(t_n) \overline{\left( \nabla u(t_n) - \nabla u^n  \right)} + \int_{\R^d} \overline{\nabla u^n} \left( \nabla u(t_n) - \nabla u^n  \right)  \right| \\
    &  \leq \left( \| \nabla u(t_n) \|_{L^2} + \| \nabla u^n \|_{L^2} \right) \| \nabla u(t_n) - \nabla u^n \|_{L^2}
    \leq C (m_2+M_2) \tau^{\alpha}
  \end{align*}
  by the Cauchy-Schwarz inequality and applying Theorem
  \ref{theorem_convergence}. For the nonlinear part, writing that $u=\phi+v$ as
  before with $v \in H^2(\R^d)$, we first remark that
  \begin{align*}
    \| |u|^2-1 \|_{L^2} & = \| |\phi|^2-1 + 2 \Re( \overline{\phi} v) + |v|^2 \|_{L^2}
    \leq \| |\phi|^2-1 \|_{L^2} + \|\phi \|_{L^{\infty}} \| v \|_{L^2} + \| v
    \|_{L^{\infty}} \| v \|_{L^2}
  \end{align*}
  from which we deduce $|u|^2-1 \in L^2(\R^d)$ from \eqref{eq:requirements_phi}, since
  $H^2(\R^d) \subset L^\infty(\R^d)$ for $d \leq 3$. A similar bound holds
  for $u^n$ from Lemma~\ref{lem:equivalence_splittings}. This allows us to write
  that
  \begin{align*}
    &  \left| \int_{\R^d} \left(1- |u(t_n)|^2 \right)^2 - \int_{\R^d} \left(1- |u^n|^2 \right)^2 \right|
    \leq \left| \int_{\R^d} \left( |u(t_n) |^2 -|u^n|^2 \right) \left( 1-|u(t_n) |^2 + 1 - |u^n |^2  \right) \right| \\
    & \quad \leq \left( \| u(t_n) \|_{L^{\infty}} + \| u^n \|_{L^{\infty}}  \right) \left( \| 1 - |u(t_n)|^2 \|_{L^2} + \| 1 - |u^n|^2 \|_{L^2}  \right) \| v(t_n)-v^n \|_{L^2}  \leq C(\phi,m_2,M_2) \tau^{\alpha}
  \end{align*}
  as $u(t_n)-u^n=v(t_n)-v^n$ and using Proposition \ref{prop:convergence_splitting_v}. Finally for the potential part we compute
  \begin{align*}
    &\left| \int_{\R^d} V (1-|u(t_n)|^2) - \int_{\R^d} V (1-|u^n|^2)  \right|  = \left| \int_{\R^d} V u(t_n) \overline{(v(t_n)-v^n)} + \int_{\R^d} V \overline{u^n} (v(t_n)-v^n)  \right| \\
    &  \quad \leq  \| V \|_{L^2} \| v(t_n) - v^n \|_{L^2} \left( \| u(t_n) \|_{L^{\infty}} + \| u^n \|_{L^{\infty}} \right)  \leq C(\|V\|_{L^2},m_2,M_2) \tau^{\alpha}
  \end{align*}
  thanks to Proposition \ref{prop:convergence_splitting_v}, which ends the proof.
\end{proof}

\section{Numerical results} \label{sec:numerics}

We perform in this section a number of numerical tests, first on a one-dimensional dark
soliton to illustrate the convergence rates of Lie and Strang splitting schemes,
and then on the two-dimensional case with a time-dependent potential to highlight the
nucleation of quantum vortices.

\subsection{1D dark soliton}

In dimension $d=1$, as mentioned previously, an explicit solution to
\eqref{eq:GP} with $\varepsilon=1$ and $V=0$ is
given, for any $|c| < \sqrt 2$, by
\begin{equation*}
  u(t,x)=\phi_c(x+ct), \quad \phi_c(x)=\sqrt{\frac{2-c^2}{2}} \tanh \left(
    \frac{\sqrt{2-c^2}}{2} x  \right) + i \frac{c}{\sqrt{2}}.
\end{equation*}
This explicit non-trivial solution can thus be used to illustrate the convergence rate of both
splitting schemes on $u$ \eqref{eq:Lie_split_u}-\eqref{eq:Strang_split_u}
by comparing the numerical solution $u^n$, obtained from the
initial condition $u^0 = \phi_c$ to the explicit solution $u(t^n) =
\phi_c(\cdot + ct_n)$. In what follows, we choose $c=1.3$. The next simulation is
performed with finite differences in a box $[-L,L]$, with $L=20$ and with
Dirichlet boundary conditions given by the values of $\phi_c$ at $\pm \infty$.
We present in Figure~\ref{fig:dark_soliton} both the numerical approximation
$u^n$ and the reference solution at time $t=0$ and $t=T=2$. At least visually,
the two solutions seem to coincide.

\begin{figure}[h!]
  \includegraphics[width=0.59\linewidth]{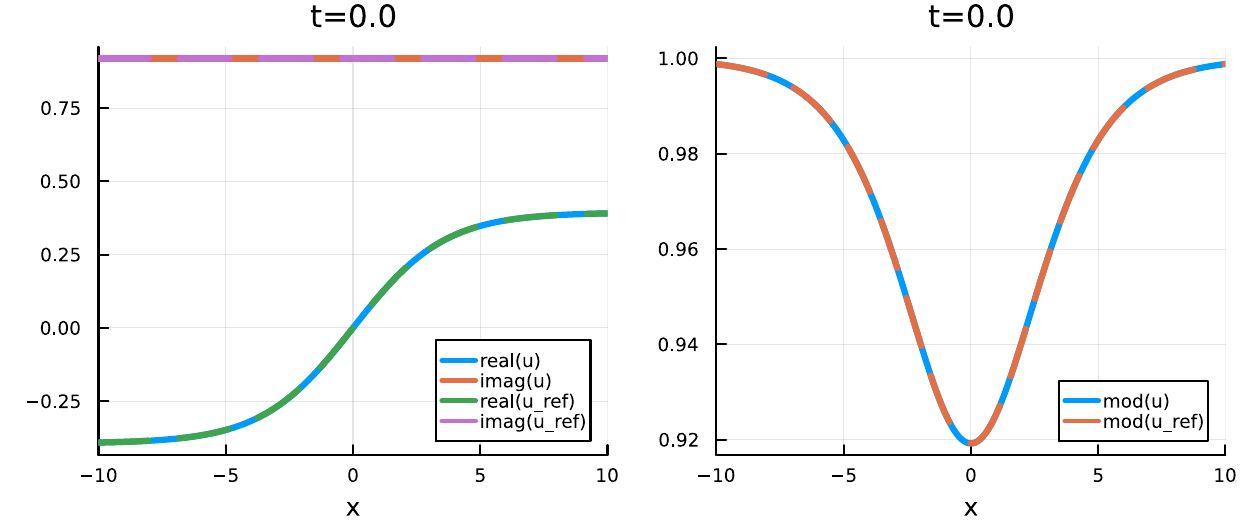}  
  \includegraphics[width=0.59\linewidth]{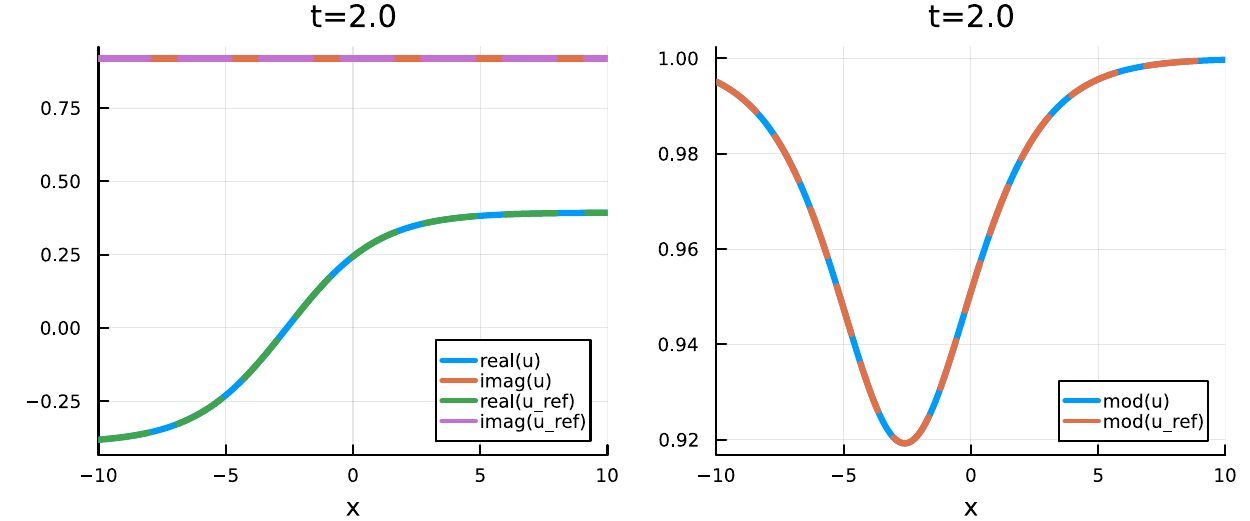}
  \caption{Exact (\texttt{u\_{ref}}) and approximate (\texttt{u}) solutions to
    \eqref{eq:GP} in one dimension, obtained with the Strang splitting scheme,
    at time $t=0$ (top) and $t=2$ (bottom). The spatial domain has been
    restricted to $[-10,10]$ for plot clarity.}
  \label{fig:dark_soliton}
\end{figure}

Next, we illustrate the convergence rates of the Lie and Strang splitting
schemes claimed in Theorem~\ref{theorem_convergence}.  For this experiment, we
set $L=60$ large enough to neglect side effects. First, we display in
Figure~\ref{fig:1D_psv} the near-preservation of the energy together with the
preservation of the mass as time evolves. The energy is not preserved, with a
drift at the final time $T=1$ smaller with $\tau$, while the mass seems to
be effectively preserved, up to the numerical accuracy of $10^{-10}$.

In Figure~\ref{fig:1D_cvg}, we plot the convergence of the $X^2$ error at final
time $T=1$, \emph{i.e.} $\| u^N - u(T)\|_{X^2}$, with respect to
$\tau$, as well as the energy error, \emph{i.e.} $|\mathcal{E}(u^N) -
\mathcal{E}(u(T))|$. As expected, for the $X^2$ norm, the Lie scheme exhibits a
first order convergence while the Strang scheme is second order. However, the
energy seems to be preserved within an order given by twice the order of the
error in the $X^2$ norm. This super-convergence in $\tau$ of the energy can be
explained by the initial condition being the dark soliton $u_0(x) = \phi_c(x)$,
as its particular symmetries allow for a compensation between the kinetic and
the nonlinear part of the Ginzburg-Landau energy. Indeed, we performed the same
experiments with initial condition $u_0(x) = \phi_c(x) - \frac12\exp(-x^2)$ and
a reference solution computed with $\tau_{\rm ref} = 5\cdot10^{-5}$. This time,
the $X^2$ error behaves similarly while the energy at $T=1$ is preserved with
the expected order (namely at first order for the Lie splitting scheme and at
second order for the Strang splitting scheme), see Figure~\ref{fig:1D_sharp},
showing that the result from Proposition~\ref{prop:nrj_preserv} is sharp.

\begin{figure}[h!]
  \includegraphics[width=0.8\linewidth]{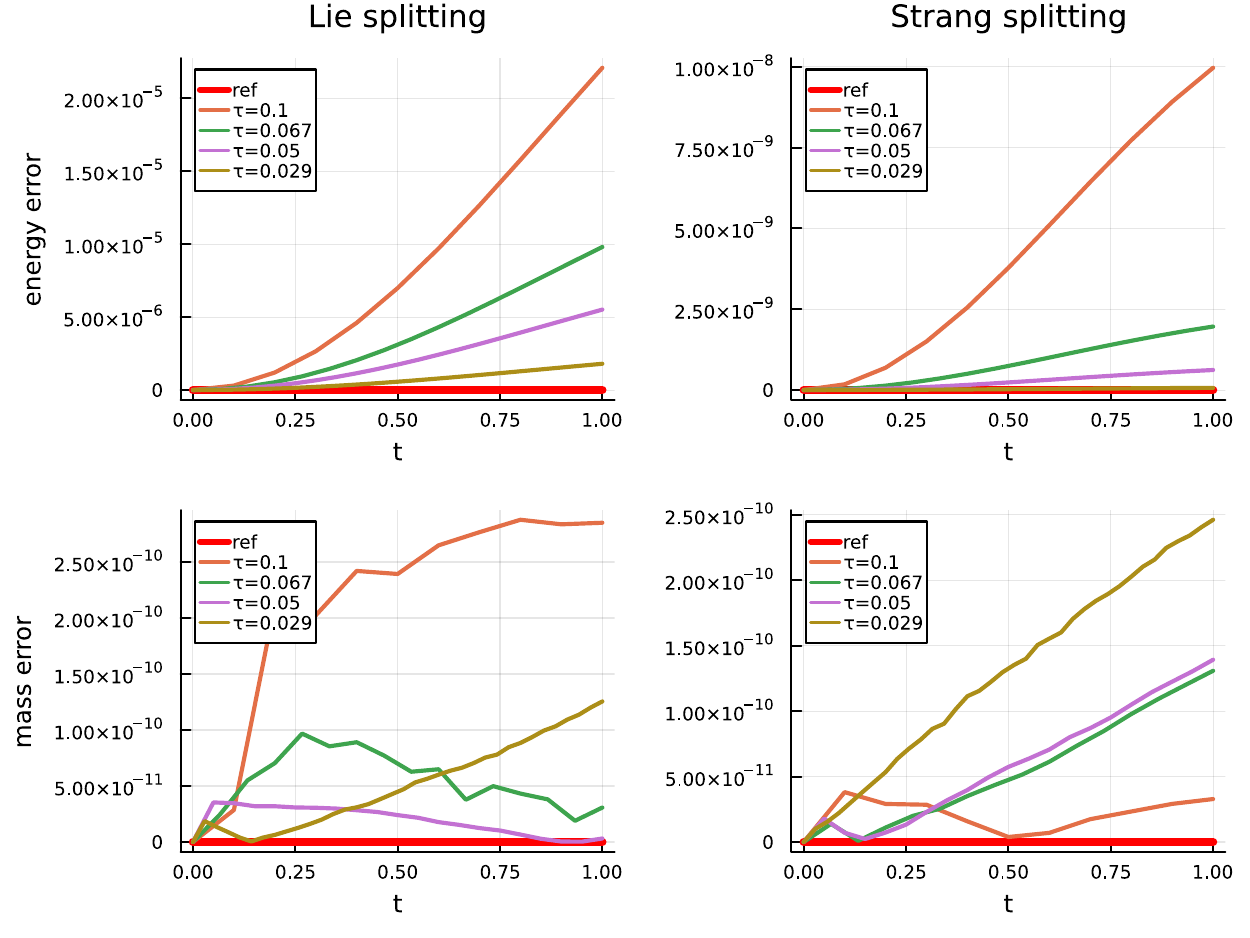}
  \caption{Visual near-preservation of the energy (top) and preservation of the
    mass -- up to numerical accuracy -- for both the Lie (left) and Strang (right)
    splitting schemes.}
  \label{fig:1D_psv}
\end{figure}

\begin{figure}[p!]
  \vspace{2cm}
  \includegraphics[width=0.8\linewidth]{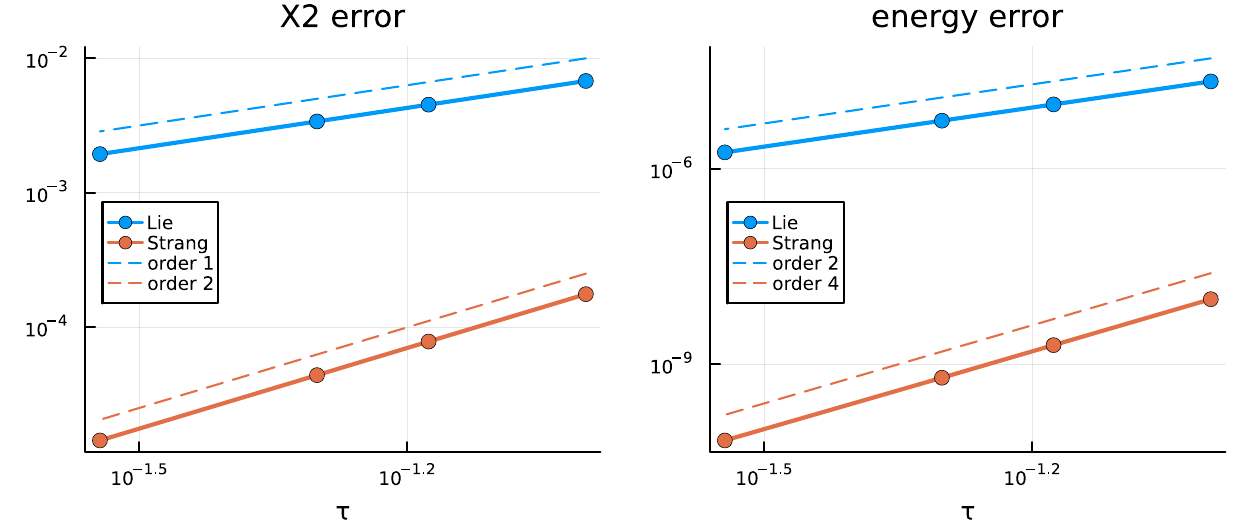}
  \caption{Case where $u_0(x) = \phi_c(x)$ is given by a dark soliton.
    (Left) Convergence in $\tau$ of both Lie and Strang splitting
    schemes. The error is computed as $\|u^N - u(T)\|_{X^2}$, where $u^N$ is the
    approximation at time $T$ and $u(T)$ is the exact solution.
    (Right) Super-convergence in $\tau$ of the energy for both
    Lie and Strang splitting schemes. The error is computed as
    $|\mathcal{E}(u^N) - \mathcal{E}(u(T))|$.}
  \label{fig:1D_cvg}
\end{figure}
\begin{figure}[p!]
  \vspace{2cm}
  \includegraphics[width=0.8\linewidth]{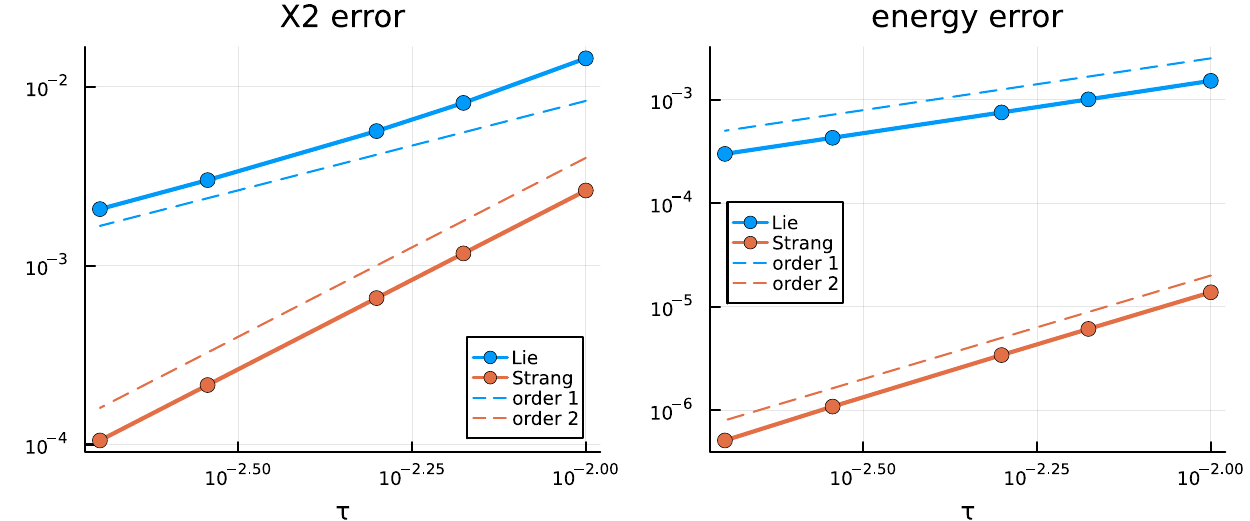}
  \caption{Case where $u_0(x) = \phi_c(x) - \frac12\exp(-x^2)$.
    (Left) Convergence in $\tau$ of both Lie and Strang splitting
    schemes. The error is computed as $\|u^N - u(T)\|_{X^2}$, where $u^N$ is the
    approximation at time $T$ and $u(T)$ is the exact solution.
    (Right) Convergence in $\tau$ of the energy for both
    Lie and Strang splitting schemes. The error is computed as
    $|\mathcal{E}(u^N) - \mathcal{E}(u(T))|$. This time, we observe the expected
   order.}
  \label{fig:1D_sharp}
\end{figure}

\clearpage
\subsection{2D vortex nucleation}
We now place our work within its framework of physical relevance, which is the
theory of Bose-Einstein condensation and quantum turbulence, and present an
application of our numerical scheme in such a setting. More precisely, we now investigate the vortex nucleation process for quantum superfluids whose dynamics
is governed by \eqref{eq:GP} in dimension two. We consider two different
time-dependent potentials $V$, with fixed positive constants for the amplitude
$V_0$, the velocity $a$, the localization $\gamma$ and the span $r_0$ of the
potential, writing $x=(x_1,x_2)\in \R^2$:
\begin{description}
	\item[Case (i)] a linearly moving Gaussian obstacle
	\begin{equation} \label{eq:potential_linear}
	V(t,x)= V_0\exp \left(-\frac{\gamma^2}{2} \left((x_1-at)^2+ x_2^2 \right) \right) ;
	\end{equation}
	\item[Case (ii)] a rotating stirring Gaussian obstacle
		\begin{equation} \label{eq:potential_rotating}
 V(t,x)= V_0\exp \left(-\frac{\gamma^2}{2} \left((x_1-r_0\cos(at))^2+(x_2-r_0\sin(at))^2)\right)  \right).
 		\end{equation}
\end{description}
Physically, we consider a scaled Laplace operator $\frac{1}{2m} \Delta$ in
\eqref{eq:GP} with atomic mass $m>0$, where $m \gg 1$ in order to restrain
dispersive effects and we use large nonlinear constant $1/\eps^2 \gg 1$ to
enhance vortex nucleation.

In both cases, we consider a large spatial periodic domain
$\T^2_L=\left[-L,L\right]^2$ of size $L=5$ to avoid any artificial boundary
effects, discretized with $\mathcal{N}$ points in both directions and integrated in
space by a Fast Fourier Transformation procedure. Time integration is performed
using the Strang splitting \eqref{eq:Strang_split_u} with time step $\tau$.
Both simulations starts from an initial state $u_0$ representing the quantum
fluid at rest, given by the computation of the global minimizer of the energy
\eqref{eq:GL} with potential $V(t=0,x)$, to reduce strong oscillations at the
start of the dynamics due to repulsive effects. Such energy minimizer is
approached by a standard gradient flow, analogously to the methods described in
\cite[Section~3]{BaoCai2013}, and solved with the LBFGS algorithm. Specific
discretization and physical parameters of each simulations are displayed in
Table \ref{table:parameters}. We then provide plots of the density and of the
phase of the numerical approximation of the solution $(u^n_{\mathcal{S}})_{0
  \leq n \leq N}$ at several times, respectively in
Figure~\ref{fig:vortices_linear_obstacle} (Case~(i)) and in
Figure~\ref{fig:vortices_rotating_obstacle} (Case~(ii)).

\begin{table}[ht]
  \centering
  \begin{tabular}{cccccccccc}
    \toprule
    & $\mathcal{N}$ & $\tau$   & $T$   & $\eps$ & $m$ &  $V_0$ & $a$ & $\gamma$ &  $r_0$ \\
    \midrule
    Case (i)   & $2^{12}$ & $10^{-4}$ & $2$ & 0.2     & 15  & $50$ & 1 & 10 & $\emptyset$ \\
    Case (ii)  & $2^{11}$ &  $4\cdot10^{-4}$ & $8$ & 0.2 & 15  & $50$ & 1 & 10 & $0.5$ \\
    \bottomrule
  \end{tabular}
  \medskip
  \caption{Numerical and physical parameters for both cases.}
  \label{table:parameters}
\end{table}

We do observe vortex nucleation for both cases in the chosen physical regimes.
In Figure~\ref{fig:vortices_linear_obstacle}, after a transient state where we
only see wave propagation in front of the potential, we observe the
sought-after nucleation of vortex pairs behind the defect, in a similar way than
in \cite[Figure 3]{Glorieux2025}. Some vortex/anti-vortex pairs first form in a conic-shaped trail of the condensate (see time $t=0.8$ in Figure~\ref{fig:vortices_linear_obstacle}), but after some time we observe that vortex/anti-vortex pairs also nucleate symmetrically with respect to the $x_1$ axis
(see time $t=2.0$ in Figure~\ref{fig:vortices_linear_obstacle}), as already highlighted by previous
studies \cite{HuepeBrachet00}.
In Figure~\ref{fig:vortices_rotating_obstacle}, vortex pairs periodically
nucleate and go straight \emph{outside} of the disc of radius $r_0$, once they
have emerged in the quantum superfluid. Vortex pairs also nucleate
\emph{inside} the disc, and start to interact together (one can observe exchange
of vortices between vortex/anti-vortex pairs). This is reminiscent of the
periodic formation of traveling waves for the one-dimensional Gross-Pitaevskii
flow past an obstacle \cite{Hakim1997}, and also observed in a two-dimensional
settings \cite{HuepeBrachet00}.

\begin{figure}[p!]
  \includegraphics[width=\linewidth]{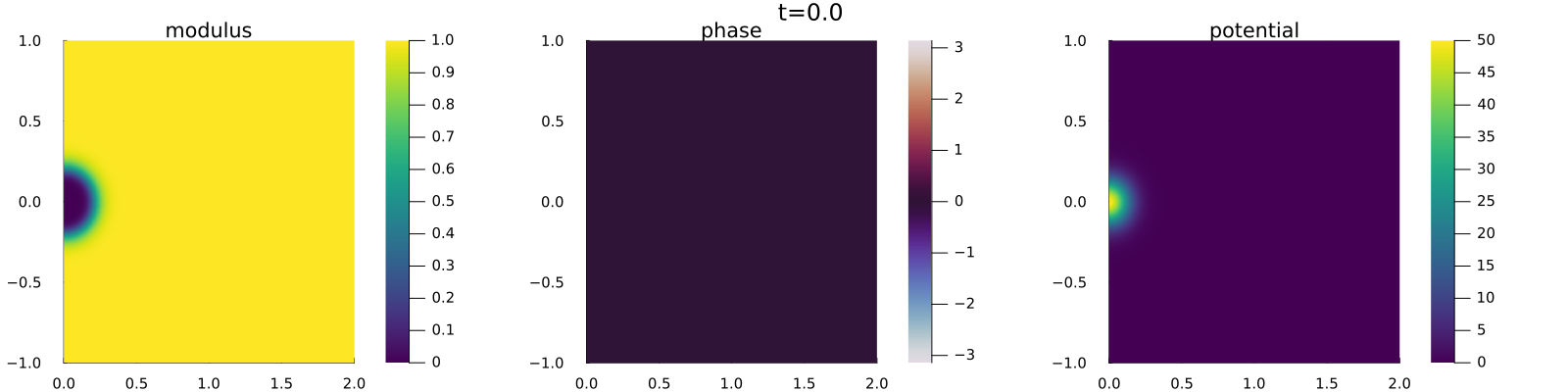}
  \includegraphics[width=\linewidth]{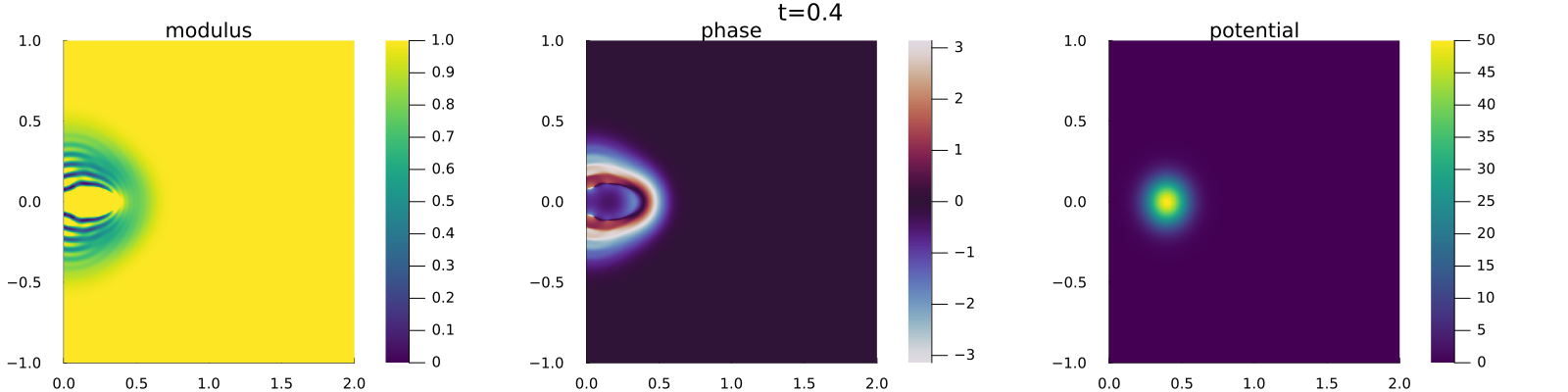}
  \includegraphics[width=\linewidth]{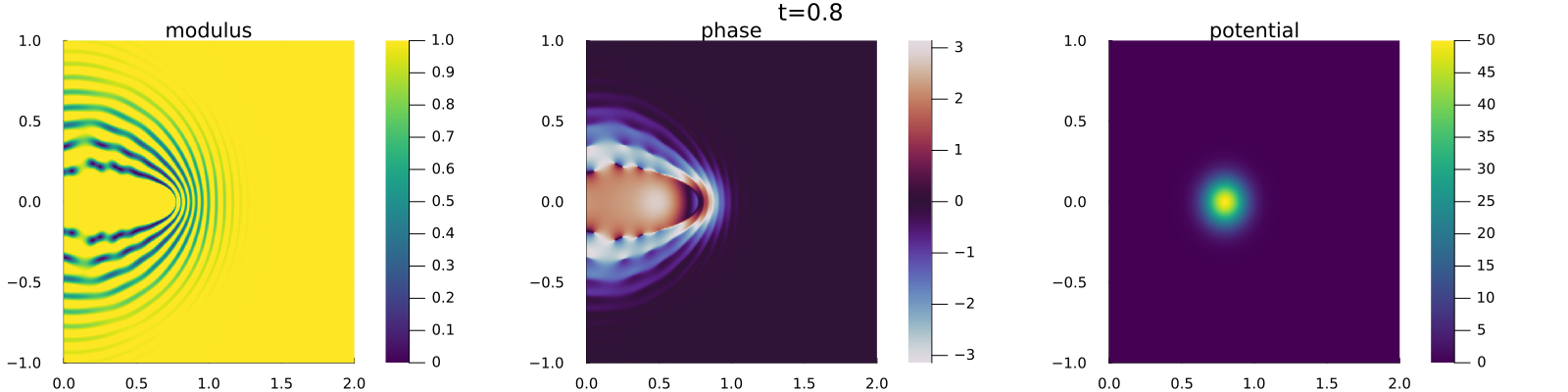}
  \includegraphics[width=\linewidth]{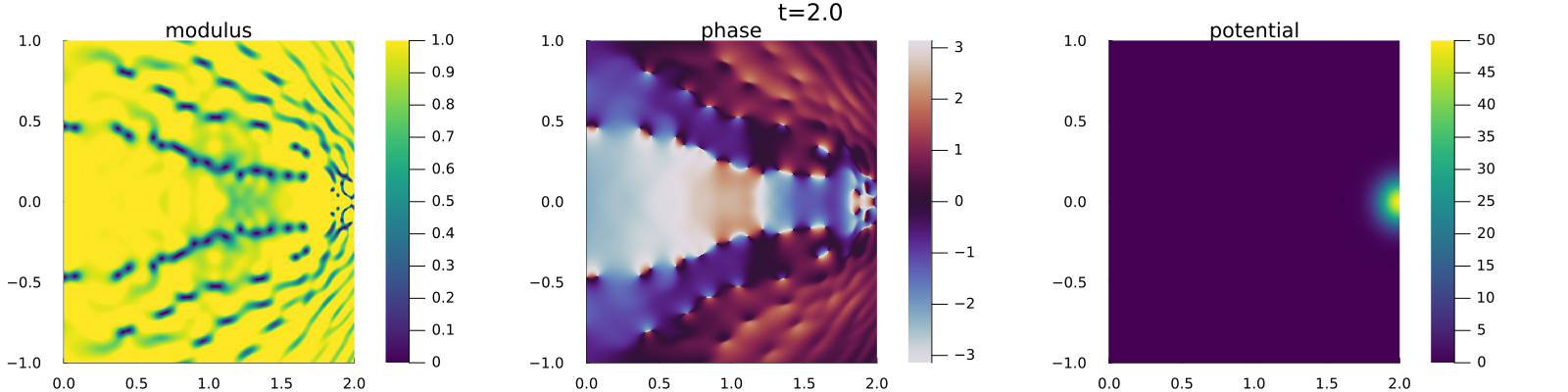}
  \includegraphics[width=\linewidth]{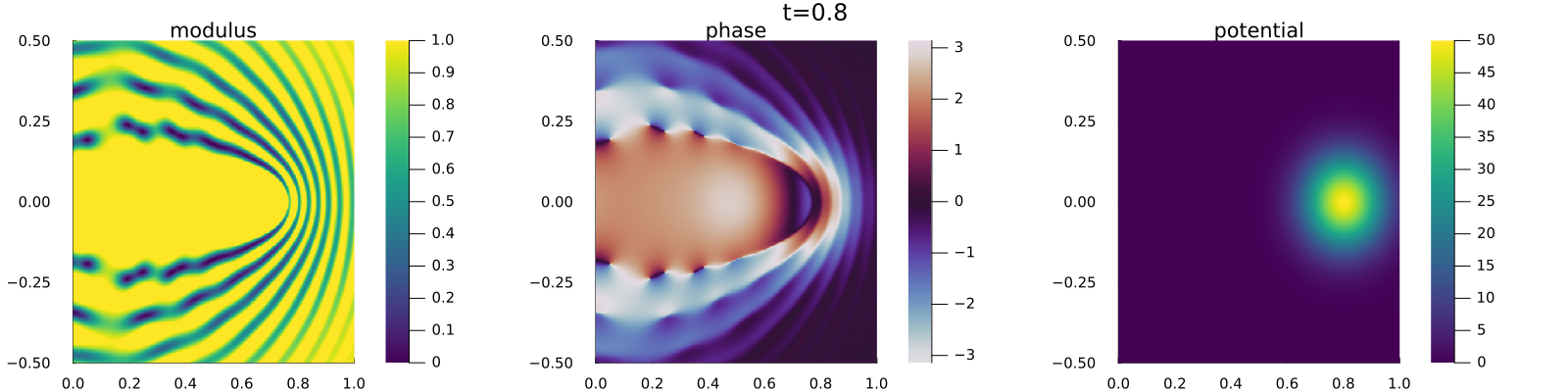}
  \caption{Numerical simulation by the Strang splitting scheme \eqref{eq:Strang_split_u}
    of the solution $u$ to \eqref{eq:GP} with a linearly moving
    Gaussian obstacle \eqref{eq:potential_linear} (Case~(i)) and parameters
    given in Table \ref{table:parameters}. (Left) density $|u|^2$ -- (Middle)
    phase -- (Right) Potential $V(t,\cdot)$, displayed at different times $t=0,
    0.4, 0.8, 2.0$, with a zoom on the cone at $t=0.8$ (bottom line).}
  \label{fig:vortices_linear_obstacle}
\end{figure}

\begin{figure}[p!]
  \includegraphics[width=\linewidth]{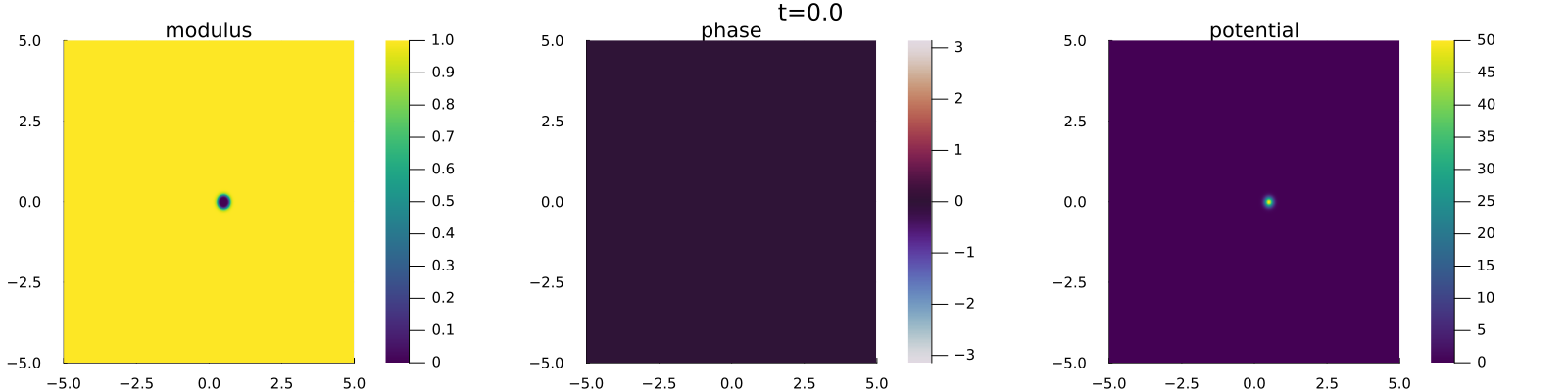}
  \includegraphics[width=\linewidth]{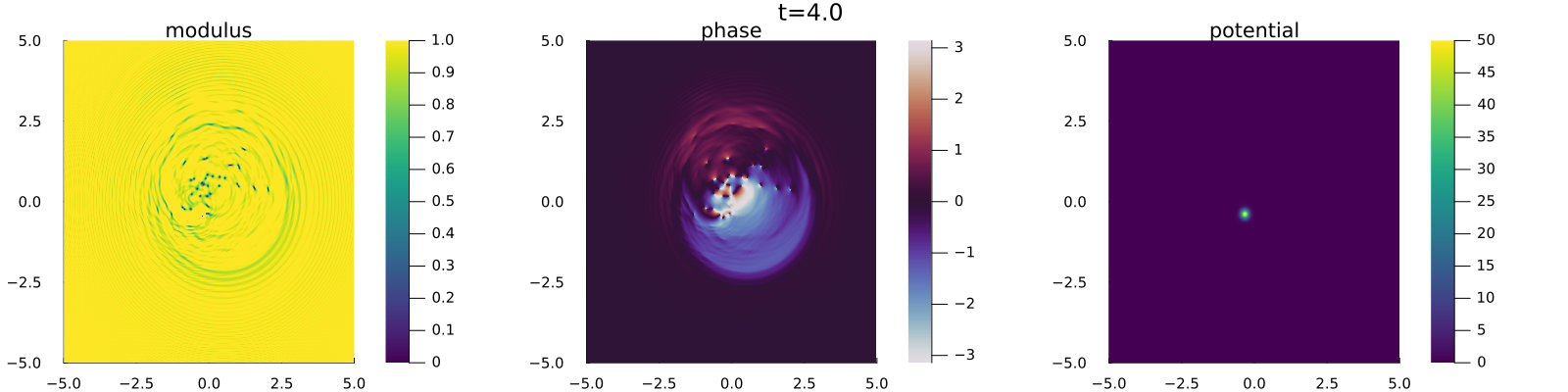}
  \includegraphics[width=\linewidth]{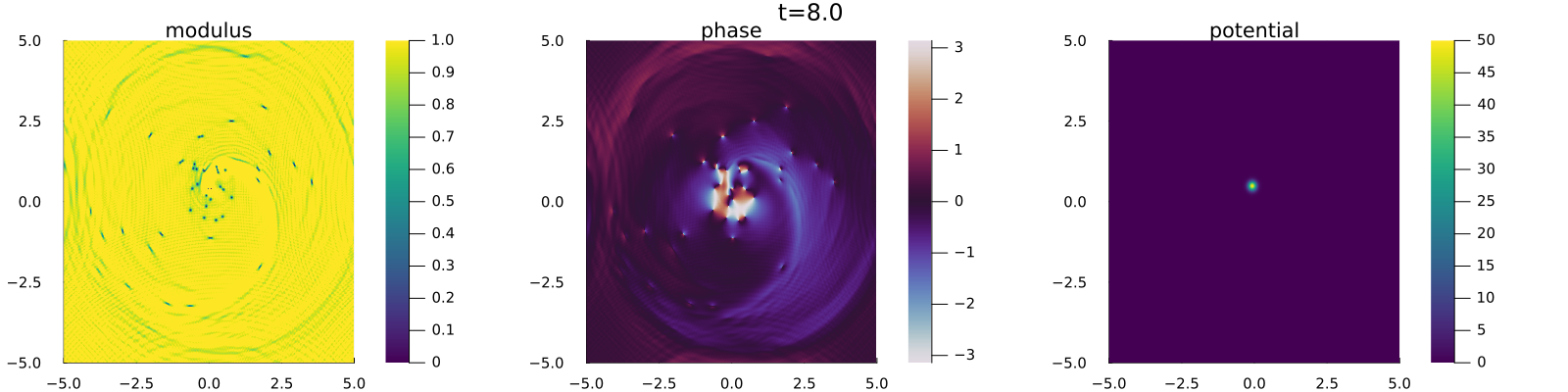}
  \includegraphics[width=\linewidth]{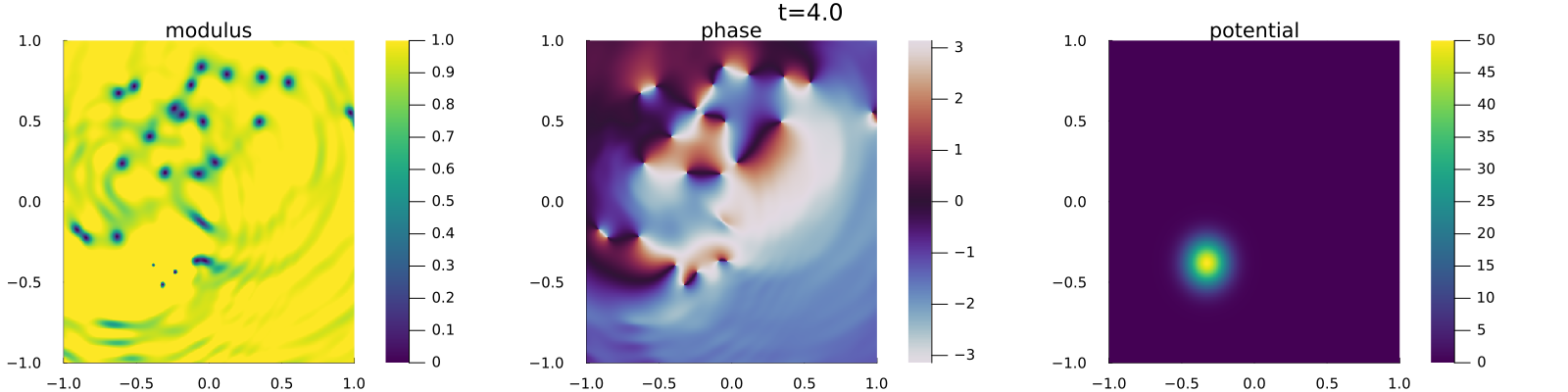}
  \includegraphics[width=\linewidth]{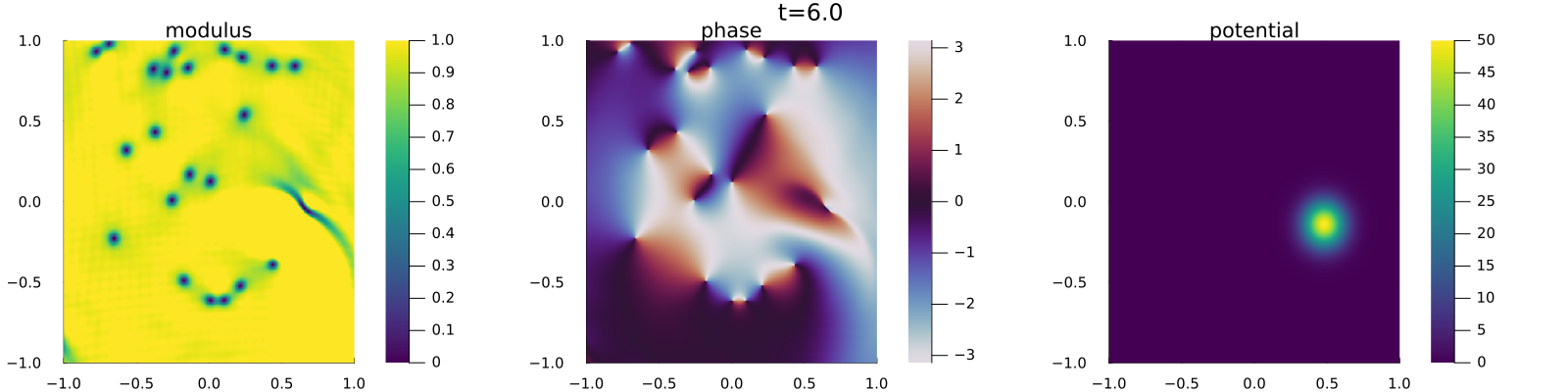}
  \caption{Numerical simulation by the Strang splitting scheme \eqref{eq:Strang_split_u}
    of the solution $u$ to \eqref{eq:GP} with a rotating Gaussian
    obstacle \eqref{eq:potential_rotating} (Case~(ii)) and
    parameters given in Table \ref{table:parameters}. (Left) density $|u|^2$ --
    (Middle) phase -- (Right) Potential $V(t,\cdot)$, displayed at different
    times $t=0, 4.0, 8.0$, with a zoom around the origin at $t=4.0, 6.0$ (two
    bottom lines).}
  \label{fig:vortices_rotating_obstacle}
\end{figure}

\clearpage
\section*{Data availability}

All the simulations have been performed with a custom \texttt{Julia} code, and
can be reproduced by downloading and running the scripts available at
\begin{center}
  \url{https://plmlab.math.cnrs.fr/gkemlin/splitting-gle}
\end{center}

\section*{Acknowledgements}

The authors wish to express their gratitude to André De Laire for several
helpful comments concerning traveling waves and the Cauchy theory of the
Gross-Pitaesvkii equation, and to Erwan Faou for enlightening us about the
superconvergence of the splitting schemes applied to the dark soliton. Q.C.
acknowledges the support of the CDP C2EMPI, together with the French State under
the France-2030 programme, the University of Lille, the Initiative of Excellence
of the University of Lille, the European Metropolis of Lille for their funding
and support of the R-CDP-24-004- C2EMPI project. This work was granted access to
HPC resources of \enquote{Plateforme MatriCS} within University of Picardie
Jules Verne. \enquote{Plateforme MatriCS} is co-financed by the European Union
with the European Regional Development Fund (FEDER) and the Hauts-De-France
Regional Council among others.

\appendix
\section{Technical result on the free Schrödinger flow on Zhidkov
  spaces}\label{app:zhidkov}

\begin{lemma}
  We have the identity
  \[ \frac{e^{-id\pi/4}}{\pi^{d/2}}\lim_{\delta  \to 0}\int_{\R^d} e^{(i-\delta
      )|z|^2} \dd z = 1.    \]
\end{lemma}
\begin{proof}
  Denote
  \[  I=\int_\R e^{(i-\delta )x^2} \dd x,   \]
  for all $\delta >0$. Then, making a radial change of coordinates we get
  \[  I^2 =\int_{\R^2} e^{(i-\delta )(x^2+y^2)} \dd x \dd y = 2\pi
    \int_0^{+\infty} r e^{(i-\delta )r^2} \dd r = \frac{\pi}{\delta  - i}. \]
  The quantity $I^2$ has two complex square roots, and taking the one with
  positive real part we infer $I=e^{i\pi/4} \sqrt{\pi}$. The result follows for
  $d=1$, and for $d\geq 1$ by tensorization.
\end{proof}

\section{Technical results on Lie derivatives}\label{app:lie}

We prove in this appendix a number of technical results and manipulation rules
for Lie derivatives that are used throughout the study of the local error terms
in Section~\ref{sec:local_error_taylor}. Note that most of the results are
obtained with careful applications of the chain rule and that analogous
calculations can be found, with different notations, in \cite[Appendix
A]{kochErrorAnalysisHighorder2013}.
Let us first recall that the main novelty here is that, in our case, the vector field $A$
is affine: $A(\varphi) = - i \Delta\varphi - i\Delta\phi$.
Thus the flow $\Phi_A^t$ is affine too and the
Fréchet derivative of $\varphi\mapsto \Phi_A^t(\varphi)$ is given by
$(\Phi_A^t)'(\varphi)[\psi] = e^{-it\Delta}\psi$, which we write, in short,
$(\Phi_A^t)'(\varphi) \equiv e^{-it\Delta}$.

\begin{lemma}\label{lem:app-1}
  For any two vector fields $F$ and $G$ on $H^1(\R^d)$,
  \[
    [D_F,D_G] = D_{[G,F]}.
  \]
  Moreover, if we have two vector fields $A$ and $G$ where $A$ is affine
  such that $A'(\varphi) \equiv -i\Delta$, then we have in addition
  \[
    \begin{split}
      (D_GA)'(\varphi)[\psi] &= -i\Delta( G'(\varphi)[\psi] ),\\
      (D_G A)''(\varphi) [ \psi,\chi] &= -i\Delta(
      G''(\varphi)[\psi,\chi] ).\\
    \end{split}
  \]
\end{lemma}

\begin{proof}
  These formulas come from the following calculations: for any vector field $A$ on
  $H^1(\R^d)$, it holds, for any $\varphi,\psi\in H^1(\R^d)$,
  \[
    \begin{split}
      (D_GA)(\varphi) &= A'(\varphi)[G(\varphi)], \\
      (D_GA)'(\varphi)[\psi] &= A''(\varphi)[\psi, G(\varphi)] + A'(\varphi)\left[ G'(\varphi)[\psi] \right],\\
      (D_G A)''(\varphi) [ \psi,\chi] & = A'''(\varphi)[\psi,\chi,G(\varphi)]
      + A'(\varphi)\left[G''(\varphi)[\psi,\chi]\right] \\
      & + A''(\varphi)\left[\psi,G'(\varphi)[\chi]\right]
      + A''(\varphi)\left[\chi,G'(\varphi)[\psi]\right]
    \end{split}
  \]
  where $A''(\varphi)$ is the second Fréchet derivative of $A$ at $\varphi$, which
  we recall to be a symmetric bilinear form on $H^1(\R^d) \times H^1(\R^d)$, and
  $A'''(\varphi)$ is the third Fréchet dérivative of $A$ at $\varphi$, which we
  recall to be a symmetric trilinear form on $H^1(\R^d)\times H^1(\R^d) \times
  H^1(\R^d)$.
  For the commutator, we thus have
  \[
    \begin{split}
      (D_FD_GA)(\varphi) &= (D_GA)'(\varphi)[F(\varphi)] =
      A''(\varphi)[F(\varphi), G(\varphi)] + A'(\varphi)\left[
        G'(\varphi)[F(\varphi)] \right],\\
      (D_GD_FA)(\varphi) &= (D_FA)'(\varphi)[G(\varphi)] =
      A''(\varphi)[G(\varphi), F(\varphi)] + A'(\varphi)\left[
        F'(\varphi)[G(\varphi)] \right].\\
    \end{split}
  \]
  Then, using the symmetry of the second Fréchet derivative $A''(\varphi)$, the
  associated terms cancel when computing the commutator, and we get
  \[
    \begin{split}
      ([D_F, D_G]A)(\varphi) &= (D_FD_GA)(\varphi) - (D_GD_FA)(\varphi)
      = A'(\varphi)\left[ G'(\varphi)[F(\varphi)] \right] - A'(\varphi)\left[ F'(\varphi)[G(\varphi)] \right]\\
      &= A'(\varphi)\big[ G'(\varphi)[F(\varphi)] - F'(\varphi)[G(\varphi)] \big]
      = A'(\varphi)\left[[G,F](\varphi)\right] = \left(D_{[G,F]}A\right)(\varphi)
    \end{split}
  \]
  and the first result follows.
  If $A$ is affine such that $A'(\varphi) \equiv -i\Delta$, we have in addition
  $A''\equiv0$ and $A'''\equiv0$ and thus
  \[
    (D_GA)'(\varphi)[\psi] = -i\Delta( G'(\varphi)[\psi] ) \text{ and }
    (D_G A)''(\varphi) [ \psi,\chi] = -i\Delta( G''(\varphi)[\psi,\chi] ).
  \]
\end{proof}

\begin{lemma}\label{lem:app0}
  Let $A,B,C$ be three vector flows on $H^1(\R^d)$ such that
  $(\Phi_A^t)'(\varphi)\equiv e^{-it\Delta}$. Then, for
  any $u_0\in H^1(\R^d)$, and any $s,t\geq0$,
  \[
    \exp(sD_C)D_B\exp(tD_A)\Id(u_0) = e^{-it\Delta}(B(\Phi_C^s(u_0))) .
  \]
\end{lemma}

\begin{proof}
  Since $(\Phi_A^t)'(\varphi) \equiv e^{-it\Delta}$ for any $\varphi\in
  H^1(\R^d)$, one has
  \[
    D_B \exp(tD_A)\Id(u_0) = D_B\Phi_A^t(u_0) = e^{-it\Delta}(B(u_0)).
  \]
  Applying \eqref{eq:Taylor_Lie}, we then get
  \[
    \exp(sD_C) D_B \exp(tD_A) \Id(v_0) = e^{-it\Delta}(B(\Phi_C^s(u_0))).
  \]
\end{proof}

\begin{lemma}\label{lem:app1}
  Let $A$ and $G$ be two vector fields on $H^1(\R^d)$. Then,
  \[
    (\exp(tD_A)G)'(\varphi)[\psi] =
    G'(\Phi_A^t(\varphi))\big[ (\Phi_A^t)'(\varphi)[\psi]\big].
  \]
  If $A$ is affine such that $(\Phi_A^t)'(\varphi)\equiv
  e^{-it\Delta}$, then it holds
  \[
    \begin{split}
      (\exp(tD_A)G)'(\varphi)[\psi] & = G'(\Phi_A^t(\varphi))[ e^{-it\Delta}\psi], \\
      (\exp(tD_A)G)''(\varphi)[\psi,\chi]&= G''(\Phi_A^t(\varphi))[ e^{-it\Delta}\psi, e^{-it\Delta}\chi].
    \end{split}
  \]
  In particular, if $G$ is linear too (for instance, $G = {\rm Id}$), the second
  order Fréchet derivative vanishes.
\end{lemma}

\begin{proof}
  The results of the Lemma~follow from
  \[
    (\exp(tD_A)G)(\varphi) = G(\Phi_A^t(\varphi))
  \]
  by applying the chain rule
  \[
    (\exp(tD_A)G)'(\varphi)[\psi]
    = G'(\Phi_A^t(\varphi))\big[(\Phi_A^t)'(\varphi)[\psi]\big]
    = G'(\Phi_A^t(\varphi))[ e^{-it\Delta}\psi],
  \]
  where the first equality holds for any $A$ and the second only holds since $A$
  is affine. Similarly, since $(\Phi_A^t)''(\varphi) \equiv 0$ if $A$ is affine,
  only the second derivative of $G$ remains  and
  \[
    (\exp(tD_A)G)''(\varphi)[\psi,\chi]
    = G''(\Phi_A^t(\varphi))[ e^{-it\Delta}\psi, e^{-it\Delta}\chi].
  \]
\end{proof}

\begin{lemma}\label{lem:app2}
  Let $v$ be the solution to the Cauchy problem
  \[
    \begin{cases}
      v'(t) = H(v(t))\\
      v(0) = v_0\in H^1(\R^d),
    \end{cases}
  \]
  where $H = A + B$ is a vector field on $H^1(\R^d)$, with $A$ and $B$ defined
  in Section~\ref{sec:commutator_est}. Then, we have the
  following Duhamel's formula, expressed with the flows of $A$ and $B$
  or, more formally, with Lie derivatives:
  \[
    \begin{split}
      v(t) & = \Phi_A^t(v_0) + \int_0^{t} e^{-i(t-s)\Delta} B(v(s)) \dd s \\
      & =\exp(t D_A) \Id(v_0)  +  \int_0^{t} \exp(s D_H) D_B \exp((t-s)D_A) \Id(v_0) \dd s.
    \end{split}
  \]
\end{lemma}

\begin{proof}
  The first formula is the standard Duhamel's formula for nonlinear
  Schr\"odinger equations. As for the second expression,
  using the notations and rules introduced in Section~\ref{sec_Lie_derivatives},
  one can compute, for any vector field $G$,
  \[
    \begin{split}
      &\frac{\dd}{{\dd}s} \Big(\big(\exp(sD_H)\exp((t-s)D_A) G\big)(v_0)\Big)
      = \frac{\dd}{{\dd}s}\Big( G\big(\Phi_A^{t-s}(\Phi_H^s(v_0))\big)
      \Big) \\
      =&\  G'\big(\Phi_A^{t-s}(\Phi_H^s(v_0))\big)
      \Big[-A(\Phi_A^{t-s}(\Phi_H^s(v_0)))
      + (\Phi_A^{t-s})'(\Phi_H^s(v_0))[H(\Phi_H^s(v_0))] \Big].
    \end{split}
  \]
  Next, note that, for any $u_0\in H^1(\R^d)$,
  \[
    \begin{split}
      (\Phi_A^{t-s})'(u_0)[A(u_0)] &= \Big(\frac{\dd}{{\dd}\tau}
      \Phi_A^{t-s}(\Phi_A^\tau(u_0))\Big)_{\tau=0}
      = \Big(\frac{\dd}{{\dd}\tau}\Phi_A^{t-s+\tau}(u_0)\Big)_{\tau=0} \\
      &= \Big(\frac{\dd}{{\dd}\tau}
      \Phi_A^{\tau}(\Phi_A^{t-s}(u_0))\Big)_{\tau=0} = A(\Phi_A^{t-s}(u_0))
    \end{split}
  \]
  from which we deduce that, with $u_0 = \Phi_H^s(v_0)$,
  \[
    \begin{split}
      & -A(\Phi_A^{t-s}(\Phi_H^s(v_0)))
      + (\Phi_A^{t-s})'(\Phi_H^s(v_0))[H(\Phi_H^s(v_0))] \\ =&
      \ -A(\Phi_A^{t-s}(\Phi_H^s(v_0))) +
      (\Phi_A^{t-s})'(\Phi_H^s(v_0))\Big[A(\Phi_H^s(v_0)) + B(\Phi_H^s(v_0))\Big] \\ =&
      \ (\Phi_A^{t-s})'(\Phi_H^s(v_0))[B(\Phi_H^s(v_0))].
    \end{split}
  \]
  Thus,
  \[
    \frac{\dd}{{\dd}s} \Big(\big(\exp(sD_H)\exp((t-s)D_A) G\big)(v_0)\Big)
    =  G'\big(\Phi_A^{t-s}(\Phi_H^s(v_0))\big)
    \Big[(\Phi_A^{t-s})'(\Phi_H^s(v_0))[B(\Phi_H^s(v_0))]\Big].
  \]
  On the other hand, note that a direct application of \eqref{eq:Taylor_Lie} and
  Lemma~\ref{lem:app1} yields
  \[
    \begin{split}
      \big(\exp(sD_H)D_B\exp((t-s)D_A)G\big)(v_0)
      &= (D_B\exp((t-s)D_A)G)(\Phi_H^s(v_0)) \\
      &= (\exp((t-s)D_A)G)'(\Phi_H^s(v_0))\big[B(\Phi_H^s(v_0))\big] \\
      &= G'(\Phi_A^{t-s}(\Phi_H^s(v_0)))
      \Big[(\Phi_A^{t-s})'(\Phi_H^s(v_0))[B(\Phi_H^s(v_0))]\Big].
    \end{split}
  \]
  Putting things together, we obtain\footnote{One will notice that this formula
    is consistent with the formal calculation
    \[ \frac{\dd}{{\dd}s} \Big(\big(\exp(sD_H)\exp((t-s)D_A) G\big)(v_0)\Big)
      = \exp(sD_H)(D_H - D_A)\exp((t-s)D_A) = \exp(sD_H)D_B\exp((t-s)D_A).\]}
  \[
    \frac{\dd}{{\dd}s} \Big(\big(\exp(sD_H)\exp((t-s)D_A) G\big)(v_0)\Big)
    = \big(\exp(sD_H)D_B\exp((t-s)D_A)G\big)(v_0).
  \]

  Then, one computes
  \[
    \begin{split}
      \big(\exp(tD_H)G\big)(v_0) - \big(\exp(tD_A)G\big)(v_0) &=
      \Big[ \big(\exp(sD_H)\exp((t-s)D_A)G\big)(v_0) \Big]_{s=0}^t \\ &=
      \int_0^t \frac{\dd}{{\dd}s} \Big(\big(\exp(sD_H)\exp((t-s)D_A) G\big)(v_0)\Big)
      {\dd} s \\ &=
      \int_0^t \big(\exp(sD_H)D_B\exp((t-s)D_A)G\big)(v_0) \dd s.
    \end{split}
  \]
  Thus, with $G=\Id$,
  \[
    v(t) = \exp(tD_H)\Id(v_0) =
    \exp(t D_A) \Id(v_0)  +  \int_0^{t} \exp(s D_H) D_B \exp((t-s)D_A) \Id(v_0) \dd s.
  \]
  Note that, using Lemma~\ref{lem:app0}, we recover the standard Duhamel's
  formula, with $v(s) = \Phi_H^s(v_0)$,
  \[
    v(t) = \Phi_A^t(v_0) + \int_0^{t} e^{-i(t-s)\Delta} B(v(s)) \dd s.
  \]
\end{proof}

\begin{lemma}\label{lem:app3}
  Let $H,A,B,C$ be vector fields on $H^1(\R^d)$, with $(\Phi_A^t)' \equiv e^{-it\Delta}$.
  Then, we have, for any $r,s,t \geq 0$ and any $v_0\in H^1(\R^d)$ with $v(r) = \Phi_H^r(v_0)$,
  \[
    \exp(r D_H) D_B \exp(sD_A ) D_C \exp ( tD_A) \Id(v_0)
    =  e^{-it\Delta} C' \left( \Phi_A^s(v(r)) \right)
    \left[e^{-is\Delta} B(v(r)) \right].
  \]
  Moreover, for any $q\geq0$,
  \[
    \begin{split}
      & \exp(qD_H)D_B\exp(rD_A)D_B\exp(sD_A)D_B\exp(tD_A)\Id(v0)\\
      =\ & e^{-it\Delta}\Big(
      B'\left(\Phi_A^{s+r}(v(q))\right)\left[e^{-is\Delta}B'(\Phi_A^r(v(q)))
        [e^{-ir\Delta}B(v(q))]\right] \\ &+ B''\left(\Phi_A^{s+r}(v(q))\right)
      \left[e^{-i(s+r)\Delta}B(v(q)), e^{-is\Delta}B(\Phi_A^r(v(q)))\right]\Big).
    \end{split}
  \]
\end{lemma}

\begin{proof}
  The proof of the first expression follows naturally by computing all the terms in
  \[ \exp(r D_H) D_B \exp(sD_A ) D_C \exp ( tD_A) \Id(v_0) \]
  from right to left using the manipulation rules of Lie derivatives.
  For the first terms, we immediately get, from
  Lemma~\ref{lem:app0},
  \[
    \exp(sD_A) D_C \exp(tD_A) \Id(v_0) = e^{-it\Delta}(C(\Phi_A^s(v_0))).
  \]
  Next, we apply $D_B$ to the vector field $F: \varphi \mapsto
  e^{-it\Delta}(C(\Phi_A^s(\varphi)))$:
  \[
    D_B \exp(sD_A ) D_C \exp ( tD_A) \Id(v_0) = (D_BF)(v_0) = F'(v_0)[B(v_0)].
  \]
  By the chain rule and $(\Phi_A^s)'\equiv e^{-is\Delta}$, we get
  \[
    F'(\varphi)[\psi] =
    e^{-it\Delta}\Big(C'(\Phi_A^s(\varphi))\big[e^{-is\Delta}\psi\big]\Big).
  \]
  Hence,
  \[
    D_B \exp(sD_A ) D_C \exp ( tD_A) \Id(v_0) =
    e^{-it\Delta}\Big(C'(\Phi_A^s(v_0))\big[e^{-is\Delta}(B(v_0))\big]\Big).
  \]
  Finally, applying one last time \eqref{eq:Taylor_Lie} yields, with $v(r) =
  \Phi_H^r(v_0)$,
  \[
    \exp(r D_H) D_B \exp(sD_A ) D_C \exp ( tD_A) \Id(v_0)
    = e^{-it\Delta}\Big(C'(\Phi_A^s(v(r)))\big[e^{-is\Delta}(B(v(r)))\big]\Big),
  \]
  from which the result follows. The second expression then follows similarly
  to the first one, replacing $H$ by $A$ and $C$ by $B$, and then applying
  $\exp(qD_H)D_B$ to it.
\end{proof}

\begin{lemma}\label{lem:app4}
  Let $A,B$ be two vector flows on $H^1(\R^d)$. Then, for any $v_0\in
  H^1(\R^d)$ and any $r,s,t \geq0$,
  \[  \exp(t D_A) \exp(s D_B)D_B^2\Id(v_0) =
    B'(\eta) \big[B(\eta) \big],
  \]
  where $\eta = \Phi_B^s( \Phi_A^t(v_0))$.
  Moreover, if $(\Phi_A^r)'\equiv e^{-ir\Delta}$,
  \[  \exp(t D_A) \exp(s D_B)D_B^3 \exp(r D_A)\Id(v_0) = e^{-ir\Delta}\Big(
    B''(\eta)\big[B(\eta),B(\eta)\big] +
    B'(\eta)\big[B'(\eta)[B(\eta)]\big]\Big).
  \]
\end{lemma}

\begin{proof}
  First, notice that
  \[
    D_B\Id(v_0) = B(v_0) \quad\text{and}\quad D_B^2\Id(v_0) = B'(v_0)[B(v_0)].
  \]
  The first result then simply follows by applying twice \eqref{eq:Taylor_Lie}.
  For the second result, note that, since $(\Phi_A^r)'\equiv e^{-ir\Delta}$, one
  can compute
  \[
    D_B^3\exp(rD_A)\Id(v_0) = e^{-ir\Delta}\Big(
    B''(v_0)\big[B(v_0), B(v_0)] + B'(v_0)\big[B'(v_0)[B(v_0)]\big]\Big),
  \]
  and apply again \eqref{eq:Taylor_Lie} twice.
\end{proof}

\bibliographystyle{siam}
\bibliography{biblio}

\end{document}